\documentclass[12pt,thmsa, reqno]{amsart}

\usepackage{graphicx}
\usepackage{amssymb, euscript}
\usepackage[square, comma, sort&compress, numbers]{natbib}
\usepackage{amsmath}
\usepackage{multicol}
\usepackage{lineno}

\def\gnk{G_{n,k}}
\def\agnk{ G(n,k)}

\def\Cal{\mathcal}

\def\H{{\Cal H}}
\def\R{{\Cal R}}

\def\gnk{G_{n,k}}

\def\bbr{{\Bbb R}}

\def\bbc{{\Bbb C}}

\def\bbd{{\Bbb D}}
\def\bbe{{\Bbb E}}
\def\bbs{{\Bbb S}}

\def\const{{\hbox{\rm const}}}

\def\Pr{{\hbox{\rm Pr}}}

\def\gnj{G_{n,j}}
\def\agnj{G(n,j)}
\def\agnk{G(n,k)}

\def\gnk{G_{n,k}}

\def\rn{\bbr^n}

\def\part{\partial}
\def\intl{\int\limits}
\def\b{\beta}

\def\Gam{\Gamma}

\def\a{\alpha}
\def\om{\omega}

\def\Del{\Delta}

\def\vp{\varphi}

\def\g{\gamma}
\def\gam{\gamma}
\def\Lam{\Lambda}
\def\sig{\sigma}
\def\lam{\lambda}
\def\z{\zeta}
\def\e{\varepsilon}

\def\t{\tau}
\def\th{\theta}

\def\rjk{\R_{j,k}}
\def\rkj{\R_{k,j}}

\font\frak=eufm10

\def\fr#1{\hbox{\frak #1}}

\def\frF{\fr{F}}

\def\frI{\fr{I}}

\newtheorem{theorem}{Theorem}[section]
\newtheorem{lemma}[theorem]{Lemma}

\theoremstyle{definition}
\newtheorem{definition}[theorem]{Definition}
\newtheorem{example}[theorem]{Example}

\theoremstyle{remark}
\newtheorem{remark}[theorem]{Remark}

\theoremstyle{corollary}
\newtheorem{corollary}[theorem]{Corollary}

\numberwithin{equation}{section}

\newcommand{\be}{\begin{equation}}
\newcommand{\ee}{\end{equation}}

\newcommand{\bea}{\begin{eqnarray}}
\newcommand{\eea}{\end{eqnarray}}
\newcommand{\Bea}{\begin{eqnarray*}}
\newcommand{\Eea}{\end{eqnarray*}}

\def\sideremark#1{\ifvmode\leavevmode\fi\vadjust{\vbox to0pt{\vss
 \hbox to 0pt{\hskip\hsize\hskip1em
\vbox{\hsize2cm\tiny\raggedright\pretolerance10000
 \noindent #1\hfill}\hss}\vbox to8pt{\vfil}\vss}}}%

\begin{document}

\title [Radon Transforms ] {Radon Transforms for Mutually Orthogonal Affine Planes }


\author{Boris Rubin}
 \address{Department of Mathematics, Louisiana State
University,  Baton Rouge, LA, 70803, USA}
\email{borisr@lsu.edu}
\author{Yingzhan Wang}
\address{School of Mathematics and Information Science, Guangzhou University, Guangzhou 510006, China;}
\email{wyzde@gzhu.edu.cn}





\begin{abstract} We study a Radon-like  transform that takes  functions on the Grassmannian of $j$-dimensional affine planes in $\rn$ to functions on a similar manifold of $k$-dimensional planes by integration over the set of all $j$-planes that meet a given $k$-plane  at a right angle.
  The case $j=0$ gives the classical Radon-John $k$-plane transform. For any $j$ and $k$, our transform has a mixed structure combining the $k$-plane transform and the dual $j$-plane transform. The main results include action of such  transforms on rotation invariant functions, sharp existence conditions, intertwining properties, connection with Riesz potentials and inversion formulas in a large class of functions.  The consideration is inspired by the previous works  of F. Gonzalez and S. Helgason who studied the case $j+k=n-1$,  $n$ odd, on smooth compactly supported functions.

\end{abstract}
\maketitle

\section{Introduction} Let $G(n, j)$ and $G(n, k)$ be a pair of  Grassmannian bundles of affine $j$-dimensional and $k$-dimensional unoriented planes in $\rn$, respectively. In the present paper we study a Radon-like transform that takes a function $f$ on   $G(n, j)$  to a function $\rjk f$ on $G(n, k)$ when the value $(\rjk f)(\z)$ for $\z\in G(n, k)$ is defined as an integral of $f$ over the set of all planes $\t\in G(n, j)$,  which meet $\z$ at a right angle.  Our aim is to study properties of this transform and obtain explicit inversion formulas.

In the limiting case $j=0$, when  $G(n, j)$ is identified with $\rn$, our transform is the well-known Radon-John $k$-plane transform, which was studied in many books and papers; see, e.g.,  \cite {GGG, H11, Mar, Ru04} and references therein. Another limiting case $k=0$ corresponds to the dual
Radon-John  transform, which averages a given function on $G(n, j)$  over all $j$-dimensional planes passing through a fixed point $x\in \rn$. These transforms are usually  studied in parallel  with the $j$-plane transforms, but have special features; see   \cite {Go87, H11,  Ru04, So2}. Thus $\rjk f$ is a  kind of  mixture of these limiting cases. An inversion formula for the  Radon transform $\rjk f$, when $j+k=n-1$, $n$ is odd and $f\in C_c^\infty (G(n, j))$,  was obtained by Gonzalez \cite {Go84, Go87} in the form
\be\label{svkljvb} f=c\, \rkj (-\Del_{n-k})^{(n-1)/2} \rjk f,\ee
where $\rkj$ is the dual of  $\rjk$, $\Del_{n-k}$ is the Laplace operator on the fiber of the Grassmannian bundle $G(n, k)$,  and $c$ is a constant, which is explicitly evaluated; see also Helgason \cite[p. 90]{H11}, where the results from \cite {Go84, Go87} are announced.

The present paper conains new results for $\rjk f$ for all $j+k<n$ and  a large class of functions $f$. In particular, we establish sharp conditions of convergence of the integrals $\rjk f$, their relation to Riesz potentials and Radon-John transforms, and  obtain new inversion formulas.

A great deal has been written about Radon transforms on Grassmann manifolds; see, e.g.,  \cite{GGR, GGS, GK03, GK04, GR04, Gr, Gr1,  H65, Ka99, P1, Ru04a, Ru13a, RW1, RW2, Str, Zha7}. The classes of Radon transforms and the methods of these works differ from those in the present paper.
More information about  Radon transforms  and their applications can be found, e.g.,  in the books \cite {GGG, H11, Ru15}   and references therein.

\vskip 0.2 truecm

{\bf Plan of the Paper and Main Results.}  Section 2 contains necessary preliminaries. Besides  the notation, it includes basic facts  about Erd\'{e}lyi--Kober fractional integrals and derivatives, Radon-John $k$-plane transforms and Riesz potentials. In Section 3 we give precise definition of the mixed  $j$-plane to $k$-plane Radon transforms  and prove the corresponding duality relation. In Section 4 we show that these transforms on radial (i.e., rotation invariant)  functions are represented as compositions of  Erd\'{e}lyi--Kober fractional integrals and give some examples. The results of Section 4 are used in Section 5 to establish sharp conditions, under which the integral $\rjk f$ exists in the Lebesgue sense. In Section 6 we derive new formulas connecting  Radon transforms $\rjk f$ with Riesz potentials.  Section 7 is devoted to inversion formulas for $\rjk f$. Here, by the dimensionality argument, the natural setting of the problem corresponds to $j+k<n$.
The most complete information is obtained  in the following cases:

(a) any $j+k<n$, when $f$ is radial;

(b) $j+k=n-1$;

(c) any $j+k<n$, when $f$ belongs to the range of the $j$-plane Radon-John transform.

Regarding other cases, we have the following

\noindent {\bf Conjecture.} {\it  If $j+k<n-1$,  then $\rjk $ is  non-injective on the set of all infinitely smooth rapidly decreasing functions.}

\noindent {\bf Acknowledgements.} The study of the operators $\rjk f$ for all $j+k<n$   in the general $L^p$ setting was suggested by the first-named author  \cite{Ru06}, who was inspired by the  works \cite {Go84, Go87, H65}.  He is thankful to Fulton Gonzalez, Todd Quinto and Sigurdur Helgason for useful discussions  during his visit to Tufts University in April, 2006. The second-named author was supported by
National Natural Science Foundation of China,  11671414.

 \section{Preliminaries}

 \subsection {Notation}

 Let $\gnj$ and $\agnj$ be the sets of  all $j$-dimensional linear subspaces and $j$-dimensional unoriented affine planes in $\rn$, respectively.
  Each ``point'' in $\gnj$  represents a $j$-plane passing through the origin.
 Every $j$-plane $\tau\in \agnj$  is naturally parameterized by the pair
$(\xi, u)$, where $\xi \in \gnj$ and $ u \in \xi^\perp$, the
orthogonal complement of $\xi $ in $\rn$. Under this parametrization,
 the manifold  $\agnj$ is a fiber bundle with the base  $\gnj$ and the canonical projection $\pi: \t(\xi, u) \to \xi$. The fiber $\pi^{-1}\xi$ over the point $\xi \in \gnj$ is the set of all $j$-dimensional planes parallel to  $\xi$. This set is $(n-j)$-dimensional  and indexed by $ u \in \xi^\perp$.
  We equip $\agnj$ with the product measure $d\t= d\xi du$, where $d\xi$ is the standard probability measure on $\gnj$ and $du$ is the Euclidean volume element on $\xi^{\perp}$.  Abusing notation, we write $|\t|$  for the Euclidean distance between $\t \in \agnj$ and the origin. If $\t\equiv \t (\xi,u)$, then, clearly, $|\t|=|u|=(u_1^2 + \cdots + u_n^2)^{1/2}$. Given a subspace $X$ of $\rn$, we write $G_j(X)$ and $G(j,X)$ for the Grassmannians  of all  $j$-dimensional linear subspaces and $j$-dimensional unoriented affine planes in $X$, respectively.

   In the following,  $\bbs^{n-1}$ denotes the unit sphere in $\bbr^{n}$.
  For $\th\in \bbs^{n-1}$, $d\th$ stands for the  Riemannian  measure on $\bbs^{n-1}$ so that the area of $\bbs^{n-1}$ is
$\sigma_{n-1} \equiv \int_{\bbs^{n-1}} d\th= 2\pi^{n/2} \big/ \Gamma (n/2)$.

 Let $e_1, \ldots, e_n$ be the coordinate unit
 vectors in $\rn$. We set
 \be\label{ksar} \bbe_{j}=\bbr e_1 \oplus \cdots \oplus\bbr e_{j},\qquad \bbe_{k}=\bbr e_{n-k+1} \oplus \cdots \oplus\bbr e_{n};\ee
 \be\label{oawe3}
\bbe_{\ell}=\bbr e_{j+1} \oplus \cdots \oplus\bbr e_{j+\ell}, \qquad \ell=n-j-k; \ee
\be\label{oawe1} \bbr^{n-j}=\bbr e_{j+1} \oplus \cdots \oplus\bbr e_{n},\qquad \bbr^{n-k}=\bbr e_1 \oplus \cdots \oplus\bbr e_{n-k}. \ee

The notation $O(n)$ for the orthogonal group of $\rn$  is standard.   For $\rho \in O(n)$, $d\rho$ stands for the $O(n)$-invariant probability measure on $O(n)$; $M(n)=\bbr^n  \rtimes O(n)$ is the  group of rigid motions in $\rn$.

All  integrals are understood  in the Lebesgue sense.
 The letter $c$ (sometimes with subscripts)  stands  for an unessential positive constant that may be different at each  occurrence.

 \subsection {Erd\'{e}lyi--Kober
Fractional Integrals} \label {kuku}

The following Erd\'{e}lyi--Kober type
fractional integrals on $\bbr_+ =(0, \infty)$ arise in numerous integral-geometric problems:
\bea
\label{as34b12}%
(I^{\a}_{+, 2} f)(t)
&=&\frac{2}{\Gam
(\a)}\intl_{0}^{t} (t^{2} -r^{2})^{\a-1}f (r) \, r\, dr,\\
\label{eci}
(I^{\a}_{-, 2} f)(t)
&=&\frac{2}{\Gam
(\a)}\intl_{t}^{\infty}(r^{2} - t^{2})^{\a-1}f (r) \, r\,
dr.\quad
\eea
We review basic facts about these integrals. More information can be found in
\cite[Subsection 2.6.2]{Ru15}.

\begin{lemma}
\label{lifa2}\

\textup{(i)} The integral $(I^{\a}_{+, 2} f)(t)$ is absolutely
convergent for almost all $t>0$ whenever $r\to rf(r)$ is a locally
integrable function on $\bbr_{+}$.

\textup{(ii)} If

\begin{equation} %
\label{for10z}
\intl_{a}^{\infty}|f(r)|\, r^{2\a-1}\, dr <\infty,\qquad a>0,
\end{equation}
then $(I^{\a}_{-, 2} f)(t)$ is finite for almost all $t>a$. If
$f$ is non-negative, locally integrable on $[a,\infty)$, and
(\ref{for10z}) fails, then $(I^{\a}_{-, 2} f)(t)=\infty$ for every
$t\ge a$.
\end{lemma}

Fractional derivatives of the Erd\'{e}lyi--Kober type are defined as the
left inverses ${\Cal D^{\a}_{\pm, 2} = (I^{\a}_{\pm,
2})^{-1}}$ and have different analytic expressions. For example, if
$\alpha= m + \alpha_{0}, \; m = [\alpha], \; 0 \le\alpha_{0} < 1$,
then, formally,
\begin{equation} %
\label{frr+z}
\Cal D^{\a}_{\pm, 2} \vp=(\pm D)^{m +1}\, I^{1 - \alpha_{0}}_{
\pm, 2}\vp, \qquad D=\frac{1}{2t}\,\frac{d}{dt}.
\end{equation}
More precisely,  the following statements hold.

\begin{theorem} \label{78awqe555} Let $\vp= I^{\a}_{+, 2} f$, where $rf(r)$ is locally integrable on $\bbr_{+}$.
Then $f(t)= (\Cal D^{\a}_{+, 2} \vp)(t)$ for
almost all $t\in\bbr_{+}$, as in (\ref{frr+z}).
\end{theorem}

\begin{theorem}
\label{78awqe} If $f$ satisfies (\ref{for10z})
for every $a>0$ and
$\vp\!= \!I^{\a}_{-, 2} f$, then $f(t)= (\Cal D^{\a}_{-, 2} \vp)(t)$ for
almost all $t\in\bbr_{+}$, where $\Cal D^{\a}_{-, 2} \vp$ can be represented as follows.

\noindent
\textup{(i)} If $\a=m$ is an integer, then
\begin{equation} %
\label{90bedr}
\Cal D^{\a}_{-, 2} \vp=(- D)^{m} \vp,
\qquad D=\frac{1}{2t}\,\frac{d}{dt}.
\end{equation}

\noindent
\textup{(ii)} If $\alpha= m +\alpha_{0}, \; m = [ \alpha], \; 0 <
\alpha_{0} <1$, then
%
\begin{equation} %
\label{frr+z33}
\Cal D^{\a}_{-, 2} \vp= t^{2(1-\a)} (- D)^{m +1}
t^{2\a}\psi, \quad\psi=I^{1-\a+m}_{-,2} \,t^{-2m-2}\,
\vp.
\end{equation}
In particular, for $\a=k/2$, $k$ odd,
\begin{equation} %
\label{frr+z3}
\Cal D^{k/2}_{-, 2} \vp= t\,(- D)^{(k+1)/2} t^{k}I^{1/2}_{-,2}
\,t^{-k-1}\,\vp.
\end{equation}
\end{theorem}

Fractional integrals and derivatives of the Erd\'{e}lyi--Kober type possess the semi-group property
\be\label{kauky}
\Cal D^{\a}_{\pm, 2}\Cal D^{\b}_{\pm, 2} = \Cal D^{\a+\b}_{\pm, 2}, \qquad I^{\a}_{\pm, 2}I^{\b}_{\pm, 2} = I^{\a+\b}_{\pm, 2}\ee
in suitable classes of functions, which guarantee the existence of the corresponding expressions.

 \subsection {Radon-John $k$-Plane Transforms and Riesz Potentials}\label {wqwafc678}
 We recall some facts from \cite{Ru13, Ru04}; see also \cite{GGG, H11}.
 The {\it $k$-plane transform} of a function $f$ on $\rn$  is defined by the formula
\be (R_k f)(\z)= \intl_{\z} f(x)\, d_\z x, \qquad  \z \in G(n,k), \quad 1\le k\le n-1,\ee
where $ d_{\z} x$ stands for the Euclidean measure on $\z$. Using parametrization $\z\equiv \z (\eta,v)$,  $\eta\in \gnk$, $ v\in \eta^{\perp}$, we have
\be\label{rtra1kfty} (R_k f)(\z) \equiv (R_k f)(\eta, v) = \intl_\eta f(v+y)\,d_\eta y,\ee
where $d_\eta y$ is the Euclidean volume element on $\eta$.
The {\it dual
$k$-plane transform} $R_k^*$  averages a function $\vp$ on $G(n,k)$  over all $k$-planes  passing through the fixed point $x\in \rn$.   Specifically,
\be  \label{mmsdcrt} (R_k^* \vp)(x)=\intl_{O(n)}\vp(\g\eta_0 +x) \,d\g, \ee
where $\eta_0$ is an arbitrary fixed $k$-plane through the origin.
The  duality relation
 \be \label{khryasdu} \intl_{G(n,k)} (R_k f)(\z)\, \vp (\z)\, d\z = \intl_{\rn} f (x)\, (R_k^* \vp)(x)\, dx
\ee holds provided that either side of this equality exists in the Lebesgue sense.

The following inequality, which is a particular case of Lemma 2.6 from \cite{Ru04}, shows for which $f$ the integral $(R_k f)(\z)$ is well defined for almost all $\z\in G(n,k)$. We have
\be\label{SDIvyh34}
\intl_{G(n,k)} \frac{|(R_k f)(\z)|}{1+|\z|}\,|\z|^{k-n+1}\, d\z\le c\, \intl_{\rn} \frac{|f(x)|}{1+|x|}\, \log (2+|x|)\, dx.\ee

 If $f$ and $\vp$ are radial functions, then  $R_k f$ and $R_k^* \vp$ can be expressed through the Erd\'{e}lyi--Kober
fractional integrals as follows.

 \begin{lemma}\label {zhujun}  {\rm (cf. \cite[Lemma 3.1, Theorem 3.2]{Ru13}, \cite[Lemma 2.1]{Ru04})} ${}$ \hfill

\vskip 0.2 truecm

\noindent {\rm (i)} If  $f(x)=f_0(|x|)$, then $(R_kf)(\z)=F_0(|\z|)$, where
\be\label{ppaawsdz}
F_0(s)= \pi^{k/2} \,(I^{k/2}_{-,2} f_0)(s).\ee

\noindent {\rm (ii)}
If  $\vp(\z)= \vp_0 (|\z|)$, then  $(R_k^* \vp)(x)= \Phi_0(|x|)$, where
\be \label {ppaawsdz1} \Phi_0(r)= \frac{\Gam (n/2)}{\Gam ((n-k)/2)}\, r^{2-n}\,(I^{k/2}_{+,2} s^{n-k-2}\vp_0)(r).\ee
The formulas (\ref{ppaawsdz}) and (\ref{ppaawsdz1}) hold provided that either side of the corresponding equality exists in the Lebesgue sense.
\end{lemma}

The duality (\ref{khryasdu}) yields the following existence result.

 \begin{lemma}\label {zhujun} {\rm (cf.  \cite[Theorem 3.2]{Ru13})} ${}$ \hfill

\vskip 0.2 truecm

 \noindent {\rm (i)}  If $f$ is locally integrable in $\bbr^n \setminus \{0\}$ and satisfies
\be \label {lkmuxk}
\intl_{|x|>a}  \,\frac{|f(x)|} {|x|^{n-k}} \, dx<\infty   \quad \text {for some} \quad a>0,\ee
 then  $(R_k f)(\z)$ is finite for  almost all $\z \in G(n,k)$.
If $f$ is nonnegative, radial, and (\ref{lkmuxk}) fails, then  $(R_k f)(\z) \equiv \infty$.

\noindent {\rm (ii)}  If $\vp \in L^1_{loc} (G(n,k)$, then  $R_k^* \vp \in L^1_{loc} (\rn)$.
 \end{lemma}

There is a remarkable connection between the operators $ R_k$, $ R^*_k$ and the Riesz potential
\be\label{rpot} (I_n^\a f)(x)=\frac{1}{\gamma_n(\a)}
\intl_{\bbr^n}
 \frac{f(y)\,dy}{|x-y|^{n-\a}},\qquad
  \gamma_n(\a)=
  \frac{2^\a\pi^{n/2}\Gamma(\a/2)}{\Gamma((n-\a)/2)},
  \ee
\[ \a >0, \qquad \; \a \neq n, n+2, n+4, \ldots \,.\]
If $f\in L^p (\rn)$, $1\le p <n/\a$, then $(I_n^\a f)(x)<\infty$ for almost all $x$, and the bounds for $p$ are sharp.

We recall that formally $I_n^\a f =(-\Del_n)^{-\a/2}f$, where $\Del_n$ is the Laplace operator in $\rn$.
The corresponding Riesz's fractional derivative is  defined as the left inverse
\be\label{sdyt} \bbd_n^\a=(I_n^\a)^{-1} \sim (-\Del_n)^{\a/2}\ee
 and has many different analytic expressions, depending on the class of functions; see \cite[Section 3.5]{Ru15} for details.
 For example, if $\vp=I_n^\a f$, $f\in L^p (\rn)$, $1\le p <n/\a$, then, by Theorem 3.44 from \cite{Ru15}, $\bbd_n^\a \vp$ can be represented as   a hypersingular integral
\bea
\label{invs1}(\bbd_n^\a
\vp)(x)&\equiv&\frac{1}{d_{n,\ell}(\a)}\intl_{\bbr^n}
\frac{(\Del^\ell_y \vp)(x)}{|y|^{n+ \a}}\, dy\\
\label{invs1a}&=&\lim\limits_{\e \to
0}\frac{1}{d_{n,\ell}(\a)}\intl_{|y|>\e} \frac{(\Del^\ell_y
\vp)(x)}{|y|^{n+ \a}}\, dy, \eea
where
\[(\Del^\ell_y \vp)(x)=\sum_{k=0}^\ell (-1)^k {\ell \choose k}
\vp(x-ky)\] is the finite difference of $\vp$ of order $\ell$
with step $y$ at the point $x$,
\bea
d_{n,\ell}(\a)&=&\frac{\pi^{n/2}}{2^\a \Gamma ((n+\a)/2)} \nonumber\\
{} \nonumber\\
&\times& \left\{ \!
 \begin{array} {ll} \! \Gamma (-\a/2) B_l (\a) \!  & \mbox{if $ \a \neq 2, 4, 6, \ldots $,}\\
{}\\
\displaystyle{\frac{2 (-1)^{\a/2-1}}{(\a/2)!} \frac {d}{d\a} B_l (\a)}, & \mbox{if $ \a = 2, 4, 6, \ldots \, ,$}\\
\end{array}
\right. \nonumber\eea
\[
B_l (\a) = \sum_{k=0}^l (-1)^{k} {l \choose k} k^{ \a}.
\]
The integer $\ell$ is arbitrary  with the choice $\ell = \a$
if $\a = 1, 3, 5, \ \ldots \ $ and any $ \ell > 2[\a/2]$ (the integer part of $\a/2$), otherwise.
 The limit in  (\ref{invs1a})
exists in the $L^p$-norm and in the almost
 everywhere sense. If, moreover,  $f$ is continuous, then the convergence in
 (\ref{invs1a}) is uniform on $\bbr^n$.

The inversion formula (\ref{invs1}) is non-local. If $\a$ is an even integer, then a local inversion formula
 $(-\Del_n)^{\a/2}I_n^\a f =f$ is available under additional smoothness assumptions for $f$; see \cite[Theorem 3.32]{Ru15} for details.

 The following theorem establishes remarkable connection between the Riesz potentials and the $k$-plane transform.

\begin{theorem}\label{howa1fu} For any   $1 \le k \le n-1$,
\be\label{kryafu} R_k^*R_kf=c_{k,n}\,I_n^{k} f, \qquad c_{k,n}=\frac{2^{k}\pi^{k/2} \Gam (n/2)}{\Gam ((n-k)/2)}, \ee
provided that either side of this equality exists in the Lebesgue sense.
\end{theorem}

The formula (\ref{kryafu}) is due to Fuglede \cite{F}. Its derivation is straightforward and relies on the Fubini theorem. The following formulas can be proved in a similar way (cf. \cite[Proposition 4.38, Lemma 4.100]{Ru15},  \cite[Section 3]{Ru04}):
\be\label{ktyu}
R_k I_n^\a f\!=\!I_{n-k}^\a R_k f, \qquad
   I_n^\a R_k^*\vp\!=\! R_k^*I_{n-k}^\a\vp.\ee
Here $I_{n-k}^\a$ stands  for the Riesz potential  on the $(n-k)$-dimensional fiber of the Grassmannian bundle $\agnk$. As above, it is assumed  that either side of the corresponding equality exists in the Lebesgue sense.
If $f$ and $\vp$ are  good enough, these formulas extend by analyticity to all complex $\a$; cf. \cite [Theorem 2.6]{Ru98}. It suffices to assume that $f$ belongs to the Semyanistyi-Lizorkin space $\Phi (\rn)$ of Schwartz functions orthogonal to all polynomials and $\vp \in R_k (\Phi (\rn))$.

\begin{corollary}  \label{kryy6}  The inversion formula for the  dual $k$-plane transform
\be\label{ktyu77} \vp = c_{k,n}^{-1} R_k \bbd_n^k R_k^*\vp, \qquad c_{k,n}=\frac{2^{k}\pi^{k/2} \Gam (n/2)}{\Gam ((n-k)/2)}, \ee
holds provided that  $ \vp = R_k f$ for some $f\in L^p (\rn)$, $1\le p <n/k$.
\end{corollary}
\begin{proof} By (\ref{kryafu}),
\[ R_k \bbd_n^k R_k^* \vp= R_k \bbd_n^k R_k^* R_k f=c_{k,n} R_k \bbd_n^k I_n^{k}f=c_{k,n} R_k f=c_{k,n} \vp.\]
\end{proof}

For further purposes, we also mention the following factorization formula, which connects the Riesz potential (\ref{rpot}) with the Erd\'{e}lyi-Kober
fractional integrals. Specifically, if $f$ is a radial function on $\rn$, $f(x)= f_0 (|x|)$, then $(I_n^\a f) (x)= F_0 (|x|)$, where
\be\label{rfa}
 F_0 (r)=2^{-\a} r^{2 - n} (I^{ \a/2}_{+,2} s^{n - \a -2} I^{ \a/2}_{-,2} f_0 )(r);\ee
see \cite{Ru83}, \cite [formula (3.4.11)]{Ru15}.
It is assumed  that either side of (\ref{rfa}) exists in the Lebesgue sense.

\subsection{The Hyperplane Radon Transform, Its Dual,  and the Funk Transform}

We will need explicit relations connecting the hyperplane Radon transform, its dual, the Funk transform, and their inverses.  Some of these facts are new. Others are scattered in the literature in different forms.

Every hyperplane $\t\in G(n, n-1)$ can be  parametrized by the pair $(\eta, t)\in \bbs^{n-1} \times \bbr$, so that
\be\label{hppl} \t\!\equiv\! \t(\eta, t)\!=\!\{x\in \bbr^n: x \cdot \eta\!=\!t\}, \qquad \t(-\eta, -t)\!=\!\t(\eta, t).\ee
 The hyperplane Radon transform $g \to Rg$ takes a function $g$ on $\rn$ to a function $Rg$ on $G(n, n-1)$ (or on $\bbs^{n-1} \times \bbr$)  by the formula
\be\label{rtra1} (Rg) (\t)\equiv (Rg) (\eta, t)=\intl_{\eta^{\perp}} g(t\eta +
y) \,d_\eta y,\ee where
$d_\eta y$ denotes the Euclidean measure on $\eta^{\perp}$. The  dual transform $h \to R^*h$ averages  a
function $h$ on $G(n, n-1)$   over the set of all hyperplanes passing through a fixed point $x\in\bbr^n$. Specifically,
\be\label{durt}
(R^*h)(x)= \frac{1}{\sig_{n-1}}\intl_{\bbs^{n-1}} h(\eta, x\cdot \eta)\,d\eta. \ee
The Funk transform of an even  function $f$ on the unit sphere $\bbs^n$ in $\bbr^{n+1}$, $n\ge 2$, is defined by the formula
\be\label{Funk.aafu}
(Ff)(\th)= \intl_{\bbs^n\cap \,\th^\perp} f(\sig) \,d_\th \sig,\qquad \th \in \bbs^n,
\ee
 where  $d_\th \sig$ stands for the standard probability measure on the $(n-1)$-dimensional sphere  $\bbs^n\cap \th^\perp$.
 The integral operator $F$ is bounded from $L^1(\bbs^n)$ to $L^1(\bbs^n)$ and every integrable even function $f$ can be explicitly reconstructed from $Ff$. A variety of different inversion formulas depending on the class of functions can be found in \cite[Section 5.1]{Ru15}; see also \cite{GGG, H11}. For example, the following statement holds.

\begin{theorem}\label{invrhys} {\rm (cf. \cite[Theorem 5.40]{Ru15})} Let $\vp=Ff$, $f\in L^1_{even} (\bbs^n)$,
 \[\Phi_\th (s)=\frac{1}{\sig_{n-1}}\intl_{\bbs^n\cap\,\theta ^\perp}
\!\!\vp(s \sig +\sqrt{1-s^2}\theta)\,d\sig, \qquad -1\le s\le 1.\]
Then
\be\label{90ashel}
f(\th) \!=  \! \lim\limits_{t\to 1}  \left (\frac {1}{2t}\,\frac {\partial}{\partial t}\right )^{n-1} \!\left [\frac{2}{(n-2)!}\intl_0^t
(t^2 \!- \!s^2)^{(n-3)/2} \,\Phi_\th (s) \,s^{n-1}\,ds\right ].\qquad\qquad\ee
In particular, for $n$ odd,
\be\label{90ashele}
f(\th) \!=  \! \lim\limits_{t\to 1} \frac{\pi^{1/2}}{\Gam (n/2)} \left (\frac {1}{2t}\,\frac {\partial}{\partial t}\right )^{(n-1)/2}[t^{n-2}\Phi_\th (t)].\ee
Altenatively, for all $n\ge 2$ we have
\be\label{90ashys}
f(\th) \!=  \! \lim\limits_{t\to 1} \, \left (\frac {\partial}{\partial t}\right )^{n-1} \left [\frac{1}{(n-2)!}\intl_0^t
(t^2 \!- \!s^2)^{(n-3)/2}\Phi_\th (s) \,s\,ds\right ].\ee
The  limit in these formulas is understood in the $L^1$-norm. If $f\in C_{even}(\bbs^n)$, it can be interpreted in the $\sup$-norm.
\end {theorem}

Numerous properties of the hyperplane Radon transform, its dual, and the Funk transform  are described in the literature; see, e.g., \cite{GGG, H11, Ru15} and references therein. Our  main concern is explicit inversion formulas for these transforms under possibly minimal assumptions for functions.

 We set
 $\rn= \bbr e_1 \oplus \cdots \oplus \bbr e_{n}$,
  \be\label{r hemisphere} \bbs^{n}_+=\{ \th=(\th_1, \ldots, \th_{n+1})\in \bbs^{n}: 0<\th_{n+1} \le 1\}.\ee
Consider the projection map
 \be\label{Con22on} \rn \ni x \xrightarrow{\;\mu\;} \th\in \bbs^{n}_+, \qquad \th=\mu (x)=\frac{x+e_{n+1}}{|x+e_{n+1}|},\ee
for which
 \be\label{Coun} x=\mu^{-1} (\th)=\frac{\th'}{\th_{n+1}}, \qquad \th' =(\th_1, \ldots, \th_{n}).\ee
 The  map $\mu$ extends to the bijection $\tilde \mu$ from the affine Grassmannian  $G(n, n-1)$  onto the set
    \be\label{Con22on5} \tilde \bbs^{n}_+=\{\om =(\om_1, \ldots, \om_{n+1})\in \bbs^n: \,  0\le \om_{n+1}<1\}.\ee
cf. (\ref{r hemisphere}). Specifically, if $\t =\t(\eta,t)\in G(n, n-1)$, $\eta\in \bbs^{n-1} \subset \rn$, $t\ge 0$, and $\tilde \t$ is the $n$-dimensional subspace containing the lifted plane $\t +e_{n+1}$, then $\om$ is  a normal vector to $\tilde \t$, so that
 \be\label{Con22on1}  \om \equiv \tilde \mu (\t)=-\eta\, \cos \, \a +e_{n+1}\sin \a, \qquad \tan \a=t.\ee

\begin{theorem} \label{Con22on22} Let
\bea\label{lafit} (Af)(\t) &=&   \frac{\sig_{n-1}}{2\sqrt{1+|\t|^2}}\, f(\tilde \mu (\t)),   \qquad  \t\in G(n, n-1),\\
\label{lafit1}  (Bg)(\th)&=&  \frac{1}{|\th_{n+1}|^n}\,  g \left(\frac{\th'}{\th_{n+1}}\right),  \qquad \th  \in \bbs^n,\quad \th_{n+1}\neq 0.\eea
If
\be\label{lafit2} \intl_{\rn} \frac{|g(x)|\, dx}{\sqrt{1+|x|^2}}< \infty,\ee
then $Bg\in L^1_{even} (\bbs^n)$ and

\be\label{lafit3} (Rg)(\t)=(AFBg)(\t),\ee
where $F$ is the Funk transform (\ref {Funk.aafu}).
\end{theorem}
\begin{proof} We make use of Theorem 5.48 from \cite{Ru15}, according to which
\be\label{lafit4}  (Ff)(\om)=(A_0 RB_0 f)(\om), \qquad \om \in  \tilde \bbs^{n}_+, \ee
where the operators $A_0$ and $B_0$ have the form
\[
(A_0h)(\om)\!=\!  \frac{2}{\sig_{n-1}\, |\om'|}\, h(\tilde \mu^{-1} (\om)),\qquad \om \!=\!(\om', \om_{n+1}), \quad  \om' \!=\!(\om_1, \ldots, \om_{n}),\]
\[
(B_0 f)(x)= (1+|x|^2)^{-n/2} f(\mu (x)), \qquad x\in \rn.\]
Extending (\ref{lafit4}) by the evenness to all $\om \in \bbs^n$ and passing to inverses, we obtain
\[
(Rg)(\t)=(A_0^{-1}FB_0^{-1} g)(\t), \qquad \t\in G(n, n-1),\]
where  $A_0^{-1}$ and $B_0^{-1}$ are defined by the formulas
\be\label{lafit5} (A_0^{-1}f)(\t)= \frac{\sig_{n-1}}{2\sqrt{1+|\t|^2}}\, f(\tilde \mu (\t))\qquad (\equiv (A f)(\t)),\ee
(here we use (\ref{Con22on1}), so that $ |\om'|^2\!=\!\cos^2 \a\!=\!(1\!+\!t^2)^{-1}\!=\!(1\!+\!|\t|^2)^{-1}$),
\be\label{lafit6} (B_0^{-1}g)(\th)= \frac{1}{|\th_{n+1}|^n}\, g \left(\frac{\th'}{\th_{n+1}}\right)  \qquad (\equiv (Bg)(\th)),\ee
(here we use (\ref{Coun}) and extend the right-hand side as an even function of $\th$). This gives (\ref{lafit3}).

The relation $Bg\in L^1_{even} (\bbs^n)$ that guarantees the existence of the Funk transform  in (\ref{lafit3}) and justifies the above reasoning, is a consequence of  (\ref{lafit2}). Indeed, using, e.g.,  \cite[formula (1.12.17)]{Ru15}, and assuming $g\ge 0$, we have
\[\intl_{\bbs^n}(Bg)(\th)\,d\th =2\intl_{\bbs^n} g \left(\frac{\th'}{\th_{n+1}}\right)\, \frac{d\th}{\th_{n+1}^n}=
2\intl_{\bbr^n}\frac{g(x)\, dx}{\sqrt{1+|x|^2}}<\infty.\]
\end{proof}

\begin{corollary}\label{corr} If $g$ satisfies (\ref{lafit2}), then the Radon transform $h(\t)=(Rg)(\t)$ exists in the Lebesgue sense for almost all $\t\in G(n, n-1)$ and can be inverted by the formula
\be\label{lafit7}
g(x)=(B^{-1}F^{-1} A^{-1}h)(x),\ee
where  the inverse Funk transform $F^{-1}$ is  evaluated by Theorem \ref{invrhys} and the operators $A^{-1}$ and $B^{-1}$ are defined by the formulas
\be\label{lafit8}
(A^{-1}h)(\om)=    \frac{2}{\sig_{n-1}\, |\om'|}\, h(\tilde \mu^{-1} (\om)),  \qquad |\om'|\neq 0,\ee
\be\label{lafit9}
(B^{-1}f)(x)=  (1+|x|^2)^{-n/2} f(\mu (x)).   \ee
\end{corollary}

Our next aim is to obtain similar statements for the dual Radon transform (\ref{durt}).

\begin{theorem} \label{Cuuip} Let
\bea\label{lafitd} (A_*f)(x)&=& \frac{1}{\sqrt{1+|x|^2}}\, f\left(\frac{-x+e_{n+1}}{\sqrt{1+|x|^2}}\right),   \qquad  x\in \rn,\\
\label{lafit1d}  (B_*h)(\th) &=& \frac{1}{|\th'|^n}\, h \left(\frac{\th'}{|\th'|},\frac{\th_{n+1}}{|\th'|} \right),  \qquad \th  \in \bbs^n,\quad |\th'|\!\neq \!0.\qquad \eea
If
\be\label{lafit2d} \intl_{G(n, n-1)} \frac{|h(\t)|\, d\t}{\sqrt{1+|\t|^2}}< \infty,\ee
then $B_* h\in L^1_{even} (\bbs^n)$ and
\be\label{lafit3d} (R^*h)(x)=(A_*FB_*h)(x),\ee
where $F$ is the Funk transform (\ref {Funk.aafu}).
\end{theorem}
\begin{proof} We make use Lemma 4.16 from \cite{Ru15}, according to which
\be\label{slvkjn}
  \!\!(R^* h)(x)=(\tilde A R\tilde Bh)(x),\ee
\bea  \label{slvkjn5} (\tilde A \tilde h)(x)&=&\frac{2}{|x|\, \sig_{n-1}}\, \tilde h\left (\frac{x}{|x|}, \frac{1}{|x|}\right ), \\
\label{slvkjn66} (\tilde Bh)(x)\!&=&\! \frac{1}{|x|^n} \,h\left (\frac{x}{|x|}, \frac{1}{|x|}\right ), \qquad x\in \rn \setminus \{0\}.\eea
Here $\tilde h$ is a  function on $G(n, n-1)$ parametrized by the pair $(\eta, t)\in \bbs^{n-1} \times \bbr_+$
 and extended to $G(n, n-1) \sim \bbs^{n-1} \times \bbr$, so that $\tilde h(-\eta, -t)\!=\!\tilde h(\eta, t)$. In our case, $\eta=x/|x|$, $t=1/|x|$.
Combining  (\ref{slvkjn}) with (\ref{lafit3}), we obtain $R^*  = \tilde A AFB\tilde B$. This formally gives (\ref{lafit3d}) with $A_*=\tilde A A$, $B_*=B\tilde B$.

Let us write
$\tilde A A$ and $B\tilde B$ in the desired form. By (\ref{slvkjn5}) and (\ref{lafit}),
\[ (A_*f)(x)=(\tilde A Af)(x)=\frac{2}{|x|\, \sig_{n-1}} (Af)\left (\frac{x}{|x|}, \frac{1}{|x|}\right )=\frac{1}{\sqrt{1+|x|^2}}\, f(\tilde \mu (\t)),\]
where $\t=\t(\eta,t)$ with  $\eta= x/|x|$, $t=1/|x|$. By (\ref{Con22on1}) with $\tan \a=t =1/|x|$, we have
\[
 \tilde \mu (\t)=-\eta\, \cos \, \a +e_{n+1}\sin \a=\frac{-x+e_{n+1}}{\sqrt{1+|x|^2}}.\]
 This gives (\ref{lafitd}). Further, by (\ref {lafit1}) and (\ref {slvkjn66}),
\[(B_*h)(\th)=(B\tilde Bh)(\th)=\frac{1}{|\th_{n+1}|^n}\,  (\tilde Bh) \left(\frac{\th'}{\th_{n+1}}\right)=\frac{1}{|\th'|^n}\, h \left(\frac{\th'}{|\th'|},\frac{\th_{n+1}}{|\th'|} \right) \]
provided that $\th_{n+1}>0$  and $|\th'|\neq 0$. Here $h (\cdot, \cdot) \equiv h(\eta, t)$ with $\eta= \th'/|\th'|$, $t=\th_{n+1}/|\th'|$. Because $h(\eta, t) =h(-\eta, -t)$, the last expression  gives (\ref{lafit1d}).

To complete the proof, it remains to show that $g=\tilde Bh$ satisfies (\ref{lafit2}). For $h\ge 0$, passing to polar coordinated and changing variables, we have
\bea
&&\intl_{\rn} \frac{(\tilde Bh)(x)\, dx}{\sqrt{1+|x|^2}}=\intl_{\rn}  \frac{1}{|x|^n} \,h\left (\frac{x}{|x|}, \frac{1}{|x|}\right )\,
\frac{dx}{\sqrt{1+|x|^2}}\nonumber\\
&&=\intl_0^\infty \frac{dr}{r \sqrt{1+r^2}}\intl_{\bbs^{n-1}} h(\eta, 1/r)\, d\eta=\intl_0^\infty \frac{dt}{\sqrt{1+t^2}}\intl_{\bbs^{n-1}} h(\eta, t)\, d\eta\nonumber\\
&&= \intl_{G(n, n-1)} \frac{h(\t)\, d\t}{\sqrt{1+|\t|^2}}< \infty.\nonumber\eea
\end{proof}

\begin{corollary}\label{cord} If $h$ satisfies (\ref{lafit2d}), then the dual Radon transform $g(x)=(R^*h)(x)$  can be inverted by the formula
\be\label{lafit77}
h(\t)=(B_*^{-1}F^{-1} A_*^{-1}g)(\t),\qquad  \t\in G(n, n-1),\ee
where  the inverse Funk transform $F^{-1}$ is  evaluated by Theorem \ref{invrhys} and the operators $A_*^{-1}$ and $B_*^{-1}$ are defined by the formulas
\bea\label{lafit88}
(A_*^{-1}g)(\om)&=&\frac{1}{|\th_{n+1}|}\, g\left(-\frac{\th'}{\th_{n+1}}\right), \qquad \th \in \bbs^n, \quad  \th_{n+1} \neq 0,\\
\label{lafit99}
(B_*^{-1}f)(\t) &=& (1\!+\!|\t|^2)^{-n/2} f(\tilde\mu (-\t)), \quad -\t\!=\!\{x\!\in \!\rn : -x\!\in \! \t\}.   \qquad \quad \eea
\end{corollary}

The analytic expressions for $A_*^{-1}$ and $B_*^{-1}$ are consequences of the formulas  (\ref{lafitd}) and (\ref{lafit1d})
 for  $A_*$ and $B_*$ and the definitions of the maps $\mu$ and $\tilde \mu$.

\begin{remark} If $h$ is infinitely differentiable and rapidly decreasing, then the inversion formula for $R^*$ can be written in the form
(\ref {ktyu77}). Specifically,
\be\label{ktyu7777} h = c_n R\, \bbd_n^{n-1} R^* h, \qquad c_n = 2^{1-n}\pi^{1-n/2}/\Gam (n/2).\ee
Here $\bbd_n^{n-1}= (-\Del_n)^{(n-1)/2}$ if $n$ is odd. If $n$ is even, then $\bbd_n^{n-1}=\Lam^{n-1}$, where $\Lam=\sum_{j=1}^n \H_j\partial_j$ is the Calderon-Zygmund operator, $\H_j$, being some singular integral operators, called  the Riesz transforms.  In this case, $\bbd_n^{n-1} R^* h \in C^\infty (\rn)$ and has the order $O(|x|^{-n})$; see Solmon \cite[p. 340]{So2} for details. The modification of  (\ref{ktyu7777}) in the form $h = c_n \bbd_1^{n-1} R R^* h$
was proved by Helgason \cite[Theorem 4.5]{H65}  in the framework of the Semyanistyi-Lizorkin spaces of  Schwartz functions orthogonal to polynomials; see also \cite[formula (4.6.38)]{Ru15}. When $n$ is odd, the formula (\ref{ktyu7777}) was obtained by Gonzalez  in \cite{Go84} (see also \cite[Theorem 2.1]{Go87}). More information on this subject, including further references,  can be found in \cite[p. 275, Notes 4.6]{Ru15}.

Our formula (\ref{lafit77}) does not contain singular integral operators and is applicable to a much larger class of functions.
\end{remark}

 \section {Mixed $j$-Plane to $k$-Plane Radon Transforms}

\subsection {Setting of the Problem}

Let  $\t\in \agnj$, $\z\in \agnk$, so that
 \[\t\equiv \t (\xi,u)=\xi +u, \qquad  \xi\in \gnj, \quad u\in \xi^{\perp},\]
\[\z\equiv \z(\eta,v)=\eta +v, \qquad  \eta\in \gnk, \quad v\in \eta^{\perp}.\]
The planes $\t (\xi,u)$ and $\z(\eta,v)$ are called perpendicular (we write $\t\perp \z$) if $a \cdot b=0$ for all $a\in \xi$ and  $b\in \eta$. We define the set of incidence
\[ \frI= \{ (\t,\z) \in  \agnj \times  \agnk : \, \t\perp \z, \quad \t\cap \z\neq \varnothing \} \]
and denote
\[ \hat \z=\{ \t \in \agnj :\;(\t,\z) \in \frI\}, \qquad \check  \t=\{ \z \in \agnk : \;(\t,\z) \in \frI\}.\]
The corresponding {\it mixed $j$-plane to $k$-plane Radon transform} has the form
\be\label{rpeeq} (\rjk f)(\z)=\intl_{\hat \z} f(\t)\, d_\z\t, \qquad (\rkj\vp)(\t)=\intl_{\check \t} \vp(\z)\, d_\t\z.\ee

\begin{remark}
It is clear that if $(\t,\z) \in \frI$, then, necessarily, $j+k\le n$, because otherwise, there exist $a\in \xi$ and  $b\in \eta$ such that $a \cdot b\neq 0$.  If $j+k= n$, then $\xi=\eta^\perp$ and
\[
(\rjk f)(\z)\equiv (\rjk f)(\eta,v)=\intl_{\eta}f(\eta^\perp+u)\,d_\eta u\quad \forall \,  v\in \eta^\perp.\]
This integral operator is in general non-injective. Indeed, if $f$ is a radial function, that is, $f(\xi, u)\equiv \tilde f (|u|)$ for some single-variable function $\tilde f$, then $\rjk f \equiv \const$  on the set of all  functions $f$ of the form $f(\xi, u)=\tilde f_\lam (|u|)=\lam ^{-k} \tilde f(|u|/\lam)$, $\lam >0$. Specifically, by rotation and dilation invariance,
\[
(\rjk f)(\z)=\intl_{\eta}\tilde f_\lam (|u|)\, du=\intl_{\bbr^{k}}\tilde f_\lam (|u|)\, du=\intl_{\bbr^{k}}\tilde f (|u|)\, du\equiv \const,\]
\[
\bbr^{k}=\bbr e_1 \oplus \cdots \oplus\bbr e_{k},\]
provided that this integral converges. Because of the lack of injectivity, the case $j+k= n$ will be excluded from our consideration and we always assume $j+k<n$.
 \end{remark}

The operators  (\ref{rpeeq}) can be explicitly  written as
\be\label{rpq}
(\rjk f)(\z)\equiv (\rjk f)(\eta,v)=\intl_{G_j(\eta^{\perp})}\,d_{\eta^{\perp}} \xi\intl_{\eta}f(\xi+u+v)\,d_\eta u,
\ee
\be\label{rpqd}
(\rkj\vp)(\t)=(\rkj\vp)(\xi,u)=\intl_{G_{k}(\xi^{\perp})}\,d_{\xi^{\perp}}\eta\intl_{\xi}f(\eta+u+v)\,d_\xi v.
\ee
Here $d_{\eta^{\perp}} \xi$ and $d_{\xi^{\perp}}\eta$ denote the canonical probability measures on the corresponding Grassmannians $G_j (\eta^\perp)$ and $G_{k}(\xi^{\perp})$,  $d_\eta u$ and   $d_\xi v$ stand for the   Euclidean measures on  $\eta$ and $\xi$, respectively.

If $j=0$ and  $k\ge 1$, then  $\rjk$ is the Radon-John transform (\ref{rtra1kfty}).
If $j\ge 1$ and  $k=0$,  then  $\rjk$ is the dual $j$-plane  transform; cf. (\ref{mmsdcrt}) with $k$ replaced by $j$.

 Our aim is to investigate the operators (\ref{rpq}) and (\ref{rpqd}) under the assumption
 \[ j>0, \qquad k>0, \qquad j+k< n.\]

\subsection{Duality}

\begin{lemma}\label{drl}
The duality relation
\be\label{dr}
\intl_{\agnk}(\rjk f)(\z)\vp(\z)\,d\z=\intl_{\agnj}f(\t)(\rkj\vp)(\t)\,d\t
\ee
holds provided that either integral exists in the Lebesgue sense.
\end{lemma}
\begin{proof}
 Let $I_r$ and $I_l$  denote the right-hand side and the left-hand side of (\ref{dr}), respectively;
\be\label{ksar} \bbe_{j}=\bbr e_1 \oplus \cdots \oplus\bbr e_{j},\qquad \bbe_{k}=\bbr e_{n-k+1} \oplus \cdots \oplus\bbr e_{n}.\ee
We set
\bea
I&=&\intl_{M(n)}f(g \bbe_j)\,\vp(g \bbe_{k})\,dg\nonumber\\
&=&\intl_{O(n)}d\rho\intl_{\rn}f(a+\rho \bbe_j)\,\vp(a+\rho \bbe_{k})\,da\,.
\eea
It suffices to show that $I=I_r=I_l$.
 Let $O_j$ and $O_k$  be  the stationary subgroups of $O(n)$   at  $\bbe_j$ and $\bbe_k$, respectively.
 We replace $\rho$ by $\rho\rho_j$  ($\rho_j\in O_j$), then integrate in $\rho_j\in O_j$, and, after that, replace $a$ by $\rho a$. Changing the order of integration and taking into account that $\rho_j \bbe_j=\bbe_j$, we obtain
\be\label {kalke}
I=\intl_{O(n)}d\rho\intl_{O_j}d\rho_j\intl_{\rn}f(\rho(a+\bbe_j))\,\vp(\rho(a+\rho_j \bbe_{k}))\,da.
\ee
Then we set $a=t+b$, where $t\in \bbe_j$,  $b\in \bbe_j^{\perp}$. Because $t+\bbe_j=\bbe_j$, (\ref{kalke}) gives
\bea
I&=&\intl_{O(n)}d\rho\intl_{\bbe_j^{\perp} }f(\rho(b+\bbe_j))db\intl_{\bbe_j}dt\intl_{O_j}\vp(\rho b+\rho (t+\rho_j \bbe_{k}))\,d\rho_j\nonumber\\
&=&\intl_{O(n)}d\rho\intl_{\bbe_j^{\perp}}f(\rho(b+\bbe_j))\,(\rkj\vp)(\rho (b+ \bbe_j))\,db=I_r\,.\nonumber
\eea
The proof of $I=I_l$ is similar.
\end{proof}

\section{Mixed Radon Transforms of Radial Functions}
We recall that a function $f$ on $\agnj$  is  radial, if there is a function $f_0$ on $\bbr_+$, such that  $f(\t)=f_0(|\t|)$.
 If $f$ is radial, then, by  (\ref{rpq}),
 \bea\label{rc}
 (\rjk f)(\eta,v)=\intl_{G_k(\eta^{\perp})}d_{\eta^{\perp}}\xi\intl_{\eta} f_0(|u+\Pr_{\xi^{\perp}}v|)\,d_\eta u,
 \eea
where $\Pr_{\xi^{\perp}}v $ denotes the orthogonal projection of $v$ onto $\xi^{\perp}$.
 The right-hand side of this equality is the $k$-plane transform of a radial function restricted to the $(n-j)$-dimensional subspace $\xi^{\perp}$ combined with the dual $j$-plane transform restricted to the $(n-k)$-dimensional space $\eta^{\perp}$.

\begin{lemma}\label{pr}
 If $f(\t)\equiv f_0(|\t|)$ satisfies the conditions
\be\label{pf4e}
\intl_0^a |f_0 (t)|\, t^{n-j-1}dt <\infty \quad \text{and} \quad \intl_a^\infty |f_0 (t)|\, t^{k-1}dt <\infty\ee
for some $a>0$, then
\be\label{pffrr} (\rjk f)(\z)=(I_{j,k} f_0)(|\z|),\ee where
\bea
(I_{j,k} f_0)(s)\!\!&=&\!\!\frac{c_1}{s^{n-k-2}}\intl_0^s(s^2\!-\!r^2)^{j/2-1} r^{\ell-1}dr\!\intl_r^{\infty}\!f_0(t)(t^2\!-\!r^2)^{k/2-1} t dt\nonumber\\
\label{pffrrt} &=&\frac{\tilde c_1}{s^{n-k-2}} (I^{j/2}_{+, 2} \, r^{\ell-2} I^{k/2}_{-, 2} f_0)(s), \quad l=n-j-k\ge 1, \eea
\be\label{pjjj4e} c_1=\frac{\sig_{j-1}\sig_{k-1}\sig_{\ell-1}}{\sig_{n-k-1}}, \qquad \tilde c_1=\frac{\pi^{k/2}\, \Gam ((n-k)/2)}{ \Gam (\ell/2)}.\ee
Moreover,
\be\label{puor} \intl_\a^\b |(I_{j,k} f_0)(s)|\,ds <\infty\quad \text{for all}\quad 0<\a<\b<\infty.\ee
\end{lemma}

\begin{proof}   It is straightforward  to show that the assumptions in (\ref{pf4e}) imply (\ref {puor}); see Appendix.   Then  $(I_{j,k} f_0)(s)<\infty $ for almost all $s>0$,  and therefore the  proof of (\ref{pffrr})  presented below is well-justified.

We transform the integral
\[
(\rjk f)(\z)\equiv (\rjk f)(\eta,v)=\intl_{G_j(\eta^{\perp})}\,d_{\eta^{\perp}} \xi\intl_{\eta}f(\xi+u+v)\,d_\eta u
\]
by changing variable  $\xi=\gamma \xi'$,  where $\gamma\in O(n)$ is an orthogonal transformation satisfying $\gamma \bbe_{k}=\eta$. Then, by (\ref{oawe1}),
\[\xi' \subset \bbe_k^\perp= \bbr^{n-k} =\bbr e_1 \oplus \cdots \oplus \bbr e_{n-k},\]
 so that  $\xi'\in G_{n-k,j}=G_j (\bbr^{n-k})$. We also set
$u=\gamma u'$, $v=\gamma v'$, where $u'\in \bbe_k$, $v' \in \bbe_k^\perp=\bbr^{n-k})$. This gives
\bea\label{hia}
(\rjk f)(\eta,v)&=&\intl_{G_j (\bbr^{n-k})} d\xi' \intl_{\bbe_{k}} (f\circ \gamma)(\xi'+u'+v')\,du'\\
&=&\intl_{G_j (\bbr^{n-k})} d\xi' \intl_{\bbe_{k}} f_0(|\xi'+u'+v'|)\,du'\nonumber\\
&=&\intl_{G_j (\bbr^{n-k})} d\xi' \intl_{\bbe_{k}} f_0  \left (\sqrt{|u'|^2 +  |\Pr_{\xi'^{\perp}}v'|^2}\,\right )\, du',\nonumber\\
&=&\sig_{k-1}\intl_{0}^\infty r^{k-1}\,dr \intl_{G_j (\bbr^{n-k})} f_0 \left (\sqrt{r^2 +  |\Pr_{\xi'^{\perp}}v'|^2}\,\right )\,d\xi' ,\nonumber\eea
where $\Pr_{\xi'^{\perp}}v'$ denotes the orthogonal projection of $v'$ onto $\xi'^{\perp}$. We set $\xi'=\a \bbe_j$, $\a\in O(n-k)$. Then
(cf. (\ref{oawe1}))
\[\xi'^{\perp} =\a\bbe_j^\perp, \qquad \bbe_j^\perp =\bbr^{n-j}=\bbr e_{j+1} \oplus \cdots \oplus \bbr e_{n},\]
and therefore
\[(\rjk f)(\eta,v)=\sig_{k-1}\intl_{0}^\infty r^{k-1}\,dr \intl_{O(n-k)} f_0 \left (\sqrt{r^2 +  |\Pr_{\a \bbr^{n-j}}v'|^2}\,\right )\,d\a.\]
Because  $v' \in \bbr^{n-k}=\bbr e_1 \oplus \cdots \oplus \bbr e_{n-k}$, we have
\[
\Pr_{\a \bbr^{n-j}}v' =\Pr_{\a \bbr^{n-j} \cap \,\bbr^{n-k} }v'=\Pr_{\a (\bbr^{n-j} \cap \,\bbr^{n-k}) }v'=\Pr_{\a\bbe_\ell} v', \]
where
\[ \bbe_\ell= \bbr e_{j+1} \oplus \cdots \oplus \bbr e_{n-k}, \qquad \ell=n-k-j.\]
Thus
\[(\rjk f)(\eta,v)=\sig_{k-1}\intl_{0}^\infty r^{k-1}\,dr \intl_{O(n-k)} f_0 \left (\sqrt{r^2 +  |\Pr_{\a\bbe_\ell}v'|^2}\,\right )\,d\a.\]
Keeping in mind that $|\Pr_{\a\bbe_\ell}v'|=|\Pr_{\bbe_\ell} \a^{-1}v'|$ and setting
 $s=|v|$, we can write the last expression  as
\[(\rjk f)(\eta,v)=\frac{\sig_{k-1}}{\sig_{n-k-1}}\intl_{0}^\infty r^{k-1}dr\intl_{\bbs^{n-k-1}}f_0(\sqrt{r^2+s^2|\Pr_{\bbe_\ell}\th|^2})\, d\th,\]
where $\bbs^{n-k-1}$ is the unit sphere in $\bbr^{n-k}$. The  inner integral can be transformed by making use of
  the bi-spherical coordinates
 \[\th=a\cos\psi+b\sin\psi, \quad a\in \bbs^{n-1}\cap \bbe_\ell,\quad b\in \bbs^{n-1}\cap \bbe_j, \quad 0<\psi<\pi/2,\] \[d\th=\sin^{j-1}\psi\cos^{\ell-1}\psi\,  da\,db\,d\psi,\]
 see, e.g., \cite[p. 31]{Ru15}.   Setting $c_1=\sig_{k-1}\sig_{\ell-1}\sig_{j-1}/\sig_{n-k-1}$, we obtain
\bea
&&\!\!\!\!\!\!\!\!\!\!(\rjk f)(\eta,v)=c_1\intl_0^\infty r^{k-1}dr\intl_0^{\pi/2}f_0\left(\sqrt{r^2+s^2\cos^2\psi}\right)\,\sin^{j-1}\psi \,\cos^{\ell-1}\psi\, d\psi\qquad \nonumber\\
&&=c_1\intl_0^\infty r^{k-1}dr\intl_0^{1}f_0\left(\sqrt{r^2+s^2\lam^2}\right)\,(1-\lam^2)^{j/2-1}\,\lam^{\ell-1}\,d\lam\nonumber\\
&&=\frac{c_1}{s^{n-k-2}}\intl_0^\infty r^{k-1}dr\intl_0^{s}f_0\left(\sqrt{r^2+t^2}\right)\,(s^2-t^2)^{j/2-1}\,t^{\ell-1}\,dt\nonumber\\
&&=\frac{c_1}{s^{n-k-2}}\intl_0^s(s^2-r^2)^{j/2-1} r^{\ell-1}dr\intl_r^{\infty}f_0(t)\,(t^2-r^2)^{k/2-1}\,t \,dt\,.\nonumber
\eea
\end{proof}

The following  analogue of Lemma \ref{pr} for the dual transform $\rkj \vp$ follows from Lemma \ref{pr} by the  symmetry.

\begin{lemma}\label{pr22}
 If $\vp(\z)\equiv \vp_0(|\z|)$ satisfies the conditions
\be\label{pf4ez}
\intl_0^a |\vp_0 (s)|\, s^{n-k-1}ds <\infty \quad \text{and} \quad \intl_a^\infty |\vp_0 (s)|\, s^{j-1}ds <\infty\ee
for some $a>0$, then
\be\label{pffrr1} (\rkj \vp)(\t)=(I^*_{j,k}\vp_0)(|\t|),\ee
 where
\bea
(I^*_{j,k}\vp_0)(t)&=&\frac{c_2}{t^{n-j-2}}\intl_0^t(t^2\!-\!r^2)^{k/2-1} r^{\ell-1}dr\!\intl_r^{\infty}\!\vp_0(s)(s^2\!-\!r^2)^{j/2-1} s \, ds\nonumber\\
\label{pffrr1t} &=&\frac{\tilde c_2}{t^{n-j-2}} (I^{k/2}_{+, 2} \, r^{\ell-2} I^{j/2}_{-, 2} \vp_0)(t), \quad l=n-j-k\ge 1,\quad \eea
\be\label{pjjj4e} c_2=\frac{\sig_{j-1}\sig_{k-1}\sig_{\ell-1}}{\sig_{n-j-1}}, \qquad \tilde c_2=\frac{\pi^{j/2}\, \Gam ((n-j)/2)}{ \Gam (\ell/2)}.\ee
Moreover,
\be\label{puor1} \intl_\a^\b |(I^*_{j,k}\vp_0)(t)|\,dt <\infty\quad \text{for all}\quad 0<\a<\b<\infty.\ee
\end{lemma}

\begin{remark} By Lemma \ref{lifa2} (ii), the finiteness of the second integrals in (\ref{pf4e}) and (\ref{pf4ez}) is necessary for the existence of the corresponding integrals $(I_{j,k} f_0)(s)$ and $(I^*_{j,k}\vp_0)(t)$.
\end{remark}

\begin{example}\label{er} The following formulas can be easily obtained from (\ref{pffrrt}) and (\ref{pffrr1t}) using tables of integrals (see, e.g., \cite{GRy}):

\vskip 0.2 truecm

\noindent {\rm (i)} If $f(\t)\!=\! |\t|^{-\lam}$, $\;k<  \lam < n-j$, then $(\rjk f)(\z)\!=\!c_{j,k} \,|\z|^{k-\lam}$, where
\be\label{prrty1}  c_{j,k}=
\frac{\pi^{k/2}\, \Gam \left (\frac{n-k}{2}\right)\, \Gam \left (\frac{\lam-k}{2}\right)\,\Gam \left (\frac{n-j-\lam}{2}\right)}
{ \Gam \left (\frac{n-j-k}{2}\right)\, \Gam \left (\frac{\lam}{2}\right)\, \Gam \left (\frac{n-\lam}{2}\right)}.\ee

\noindent {\rm (ii)} If $\vp(\z)\!=\! |\z|^{-\lam}$, $\;j< \lam < n-k$,  then  $(\rkj \vp)(\t)\!=\!c_{k,j} \,|\t|^{j-\lam}$, where
\be\label{prrty2} c_{k,j}=
\frac{\pi^{j/2}\, \Gam \left (\frac{n-j}{2}\right)\, \Gam \left (\frac{\lam-j}{2}\right)\,\Gam \left (\frac{n-k-\lam}{2}\right)}
{ \Gam \left (\frac{n-j-k}{2}\right)\, \Gam \left (\frac{\lam}{2}\right)\, \Gam \left (\frac{n-\lam}{2}\right)}.\ee

\noindent {\rm (iii)} If $f(\t)\!=\! (1+|\t|^2)^{-n/2}$, then $(\rjk f)(\z)\!=\!c_{k} \,(1+|\z|^2)^{(j+k-n)/2}$, where
\be\label{prrty3} c_{k}=
\frac{\pi^{k/2}\, \Gam \left (\frac{n-k}{2}\right)}
{ \Gam \left (\frac{n}{2}\right)}.\ee

\noindent {\rm (iv)} If $\vp(\z)\!=\! (1+|\z|^2)^{-n/2}$, then $(\rkj \vp)(\t)\!=\!c_{j} \,(1+|\t|^2)^{(j+k-n)/2}$, where
\be\label{prrty4} c_{j}=
\frac{\pi^{j/2}\, \Gam \left (\frac{n-j}{2}\right)}
{ \Gam \left (\frac{n}{2}\right)}.\ee

\end{example}

\section{Existence of the Mixed Radon Transforms}

 Example \ref{er} in conjunction with  duality  (\ref{dr}) gives information about the  existence in the Lebesgue sense of the corresponding Radon transforms $\rjk f$ and $\rkj \vp$.

\begin{theorem}\label{ODUL}  The following formulas hold provided that the integral in either side of the corresponding equality exists in the Lebesgue sense.
\be\label{dra1}
\intl_{\agnk}\frac{(\rjk f)(\z)}{|\z|^{\lam}} \,d\z=c_{k,j} \intl_{\agnj}\frac{f(\t)}{|\t|^{\lam -j}}\,d\t, \quad j<\lam < n-k;
\ee
\be\label{dra2}
\intl_{\agnk}\frac{(\rjk f)(\z)}{(1+|\z|^2)^{n/2}}\,d\z=c_{j} \intl_{\agnj}\frac{f(\t)}{(1+|\t|^2)^{(n-j-k)/2}}\,d\t;
\ee
\be\label{dra3}\intl_{\agnj} \frac{(\rkj\vp)(\t)}{|\t|^{\lam}}\, d\t=c_{j,k} \intl_{\agnk}\frac{\vp(\z)}{|\z|^{\lam -k}}\,d\z, \quad k< \lam < n-j;\ee
\be\label{dra4} \intl_{\agnj} \frac{(\rkj\vp)(\t)}{(1+|\t|^2)^{n/2}}\, d\t=c_{k} \intl_{\agnk}\frac{\vp(\z)}{(1+|\z|^2)^{(n-j-k)/2}}\,d\z.\ee
Here  $c_{k,j}, c_{j}, c_{j,k}$, and $c_{k}$ have the same meaning as in Example \ref{er}.
\end{theorem}

The next theorems characterize the existence of the Radon transforms $\rjk f$ an $\rkj \vp$ in different terms.

\begin{theorem} \label {kwwt} Let  $\;j+k< n$. If $f$ is a locally integrable function on $\agnj$ satisfying
\be\label{pfqq4e}
\intl_{|\t|>a} |f(\t)|\, |\t|^{k+j-n}\, d\t <\infty \ee
for some $a>0$, then $(\rjk f)(\z)$ is finite for almost all $\z\in G(n,k)$.
If for a nonnegative, radial,  locally  integrable function $f$, the condition (\ref{pfqq4e}) fails, then
 $(\rjk f)(\z)=\infty $  for  all $\z\in G(n,k)$.
\end{theorem}

\begin{proof}  It suffices to show that
\be\label{autyre}  I_{\a,\b} \equiv \intl_{\a<|\z|<\b} |(\rjk f)(\z)|\, d\z <\infty \quad \text{\rm for all} \quad 0<\a<\b <\infty.\ee
Because $\rjk$ commutes with orthogonal transformations,
\[
 I_{\a,\b}\equiv \intl_{O(n)}d\gam \intl_{\a<|\z|<\b} |(\rjk f)(\gam\z)|\, d\z\le \intl_{\a<|\z|<\b} (\rjk \tilde f)(\z)\, d\z,\]
where $\tilde f(\t)=\int_{O(n)} |f(\gam\t)|\,d\gam$ is a radial function. We set $\tilde f(\t)=f_0(|\t|)$. The assumptions for $f$ in the lemma imply  (\ref{pf4e}) for $f_0$. Indeed,
\bea&&\intl_0^a f_0 (t)\, t^{n-j-1}\,dt= \intl_0^a  t^{n-j-1}dt\intl_{O(n)} |f(\gam (\bbe_j +te_{j+1}))|\,d\gam\nonumber\\
&&=\intl_0^a  t^{n-j-1}\,dt\intl_{O(n-j)}d\om \intl_{O(n)} |f(\gam (\bbe_j +t\om e_{j+1}))|\,d\gam\nonumber\\
&&=\frac{1}{\sig_{n-j-1}}\intl_0^a  t^{n-j-1}\,dt \intl_{O(n)} d\gam \intl_{\bbs^{n-j-1}}  |f(\gam (\bbe_j +t\sig))|\,d\sig\nonumber\\
&&=\frac{1}{\sig_{n-j-1}}\intl_{O(n)} d\gam \intl_{y\in \bbe_j^\perp,\, |y|<a}|f(\gam (\bbe_j +y))|\, dy\nonumber\\
&&=\frac{1}{\sig_{n-j-1}}\intl_{G_{n,j}} d\xi \intl_{u\in \xi^\perp,\; |u|<a}|f(\xi +u)|\, d_{\xi^\perp} u=\frac{1}{\sig_{n-j-1}}
\intl_{|\t|<a}|f(\t)|\, d\t<\infty,\nonumber\eea
because $f$ is  locally integrable. Similarly,
\[ \intl_a^\infty f_0 (t)\, t^{k-1}dt=\frac{1}{\sig_{n-j-1}}
\intl_{|\t|>a}|f(\t)|\, |\t|^{k+j-n}\,  d\t <\infty.\]
Hence, by (\ref{pffrr}) and (\ref{puor}),
\bea
 I_{\a,\b}&\le& \intl_{\a<|\z|<\b} (\rjk \tilde f)(\z)\, d\z=\intl_{\a<|\z|<\b} (I_{j,k} f_0)(|\z|)\, d\z\nonumber\\
 &=&\sig_{n-j-1} \intl_\a^\b (I_{j,k} f_0)(s)\,s^{n-j-1}\,ds \le c  \intl_\a^\b (I_{j,k} f_0)(s)\,ds<\infty,\nonumber\eea
as desired.

To complete the proof, suppose that (\ref{pfqq4e}) fails for some
 nonnegative, radial,  locally  integrable function $f$. If $f(\t)\equiv f_0(|\t|)$ then the inner integral in   (\ref{rpq}) becomes
 \bea
 I(\xi, \eta,v)&\equiv&\intl_{\eta}f(\xi+u+v)\,d_\eta u=\intl_{\eta} f_0\left (\sqrt {|u|^2 +|\Pr_{\xi^\perp} v|^2}\right )\,d_\eta u\nonumber\\
 &=&\sig_{k-1}\intl_s^\infty f_0(t) (t^2 -s^2)^{k/2 -1}\, tdt, \qquad s= |\Pr_{\xi^\perp} v|.\nonumber\eea
 The condition (\ref{pfqq4e}) is equivalent to
 \[\intl_a^\infty f_0(t) t^{k-1}\, dt=\infty \quad \text{\rm for all} \quad a>0.\]

 By Lemma \ref {lifa2}(ii), it follows  that $I(\xi, \eta,v)=\infty$ for all triples $(\xi, \eta,v)$, and therefore  $(\rjk f)(\z)=\infty $  for  all $\z\in G(n,k)$.
\end{proof}

The following statement is an analogue of Theorem \ref{kwwt} for the dual transform $\rkj \vp$ and holds by the symmetry.
\begin{theorem} \label {kwwt32}
Let  $\;j+k<  n$. If $\vp$ is a locally integrable function on $\agnk$ satisfying
\be\label{pfqq4e5}
\intl_{|\z|>a} |\vp(\z)|\, |\z|^{k+j-n}\, d\z <\infty \ee
for some $a>0$, then $(\rkj \vp)(\t)$ is finite for almost all $\t\in G(n,j)$.
If for a nonnegative, radial,  locally  integrable function $\vp$, the condition (\ref{pfqq4e5}) fails, then
 $(\rkj \vp)(\t)=\infty $  for  all $\t\in G(n,j)$.
\end{theorem}
\begin{corollary} \label {kcf2} ${}$\hfill

\vskip 0.2 truecm

\noindent {\rm (i)} If $f\in L^p (\agnj)$, $1\le p <(n-j)/k$, then $(\rjk f)(\z)$ is finite for almost all $\z\in G(n,k)$.

\noindent {\rm (ii)} If $\vp \in L^q (\agnk)$, $1\le q <(n-k)/j$, then $(\rkj \vp)(\t)$ is finite for almost all $\t\in G(n,j)$.

\noindent The bounds  $p <(n-j)/k$ and $ q <(n-k)/j$ in these statements are sharp.
\end{corollary}
\begin{proof} (i) By Theorem \ref{kwwt}, it suffices to check (\ref{pfqq4e}). For any $a>0$, the H\"older inequality yields
\[\intl_{|\t|>a} |f(\t)|\, |\t|^{k+j-n}\, d\t \le ||f||_p \,\Bigg (\,\intl_{|\t|>a} |\t|^{(k+j-n)p'}\, d\t\Bigg )^{1/p'}, \quad \frac{1}{p} +\frac{1}{p'} =1.\]
Because the integral in bracket is finite whenever $1\le p <(n-j)/k$, the result follows. The proof of (ii) is similar. If  $p\ge (n-j)/k$,
the function
\be\label{lehget}
f(\t)=(2+|\t|)^{(j-n)/p}(\log(2+|\t|))^{-1}\ee
provides a counter-example. Indeed, this function belongs to
$L^p(\agnj)$ and does not obey (\ref{pfqq4e}). The case $ q\ge(n-k)/j$ is similar.
\end{proof}

\begin{remark} It is interesting to note that the same bounds for $p$ and $q$ can be obtained from (\ref{dra2}) and (\ref{dra4}) if we apply H\"older's inequality to the right-hand sides.
\end{remark}

\section{Mixed Radon Transforms and Riesz Potentials}

It is  known that the classical hyperplane Radon transform, its $k$-plane generalization, and their duals intertwine Laplace operators on the source space and the target space. More general intertwining formulas can be obtained if we replace the Laplace operators by the corresponding
Riesz potentials; cf. (\ref{ktyu}). Our aim in this section is to extend these formulas to the mixed Radon transforms (\ref{rpeeq}). We also obtain certain  Grassmannian analogues of Fuglede's formula (\ref{kryafu}) and its generalizations (\ref{ktyu}). Throughout this section, we keep the notation from Subsection \ref{wqwafc678}.

\subsection{Intertwining Formulas}

\begin{theorem} \label {dwalg} If $\;0<\a<n-k-j$, then
\be\label{wleiy}
(I_{n-k}^{\a} \rjk f)(\z)= (\rjk I_{n-j}^{\a} f)(\z), \qquad \z\in \agnk,\ee
provided that either side of this equality exists in the Lebesgue sense.
\end{theorem}

\begin{remark}  Before we prove this theorem, some comments are in order.

\vskip 0.2 truecm

\noindent{\bf 1.} In the limiting cases $j=0$ (the $k$-plane transform) and $k=0$ (the dual $j$-plane transform), the formula (\ref{wleiy}) agrees with (\ref{ktyu}). The corresponding formulas for the hyperplane Radon transform and its dual can be found in \cite[Proposition 4.38]{Ru15}.

\vskip 0.2 truecm

\noindent{\bf 2.} It is natural to conjecture that for sufficiently good $f$,
(\ref{wleiy}) extends by analyticity  to all complex $\a$. If $\a=-2m$, $m \in \{1,2, \ldots\}$,  then (\ref{wleiy}) reads
\be\label{wleiytt}
(-\Del_{n-k})^{m} \rjk f= \rjk (-\Del_{n-j})^{m} f,\ee
where $\Del_{n-k}$ and $\Del_{n-j}$ stand for the Laplace operators on the corresponding fibers.

\vskip 0.2 truecm

\noindent{\bf 3.}
In the case $j+k=n-1$, $m=1$, the formula
(\ref{wleiytt}) was proved by Gonzalez  under the assumption $f\in C_c^\infty (\agnk)$; cf. \cite[Lemma 3.3]{Go87}.  Similar formulas for the Radon-John transform and its dual are due to Helgason \cite[Lemma 8.1]{H65}.

\vskip 0.2 truecm

\noindent{\bf 4.} For the hyperplane Radon transform on $\rn$ and its dual, an analogue of (\ref{wleiy})  for all $\a\in \bbc$ in the corresponding Semyanistyi-Lizorkin spaces can be found in \cite[Proposition 4.58]{Ru15}.

\end{remark}

\noindent{\it Proof of Theorem \ref {dwalg}.}  We assume $ \z\equiv \z(\bbe_k, 0)$ and set
\[I=(I_{n-k}^{\a} \rjk f)(\bbe_k, 0),\qquad  J=(\rjk I_{n-j}^{\a} f)(\bbe_k, 0).\]
Because the operators in (\ref{wleiy}) commute with rigid motions, it suffices to show that $I=J$. We recall the notation for the subspaces:
\[ \bbe_{j}=\bbr e_1 \oplus \cdots \oplus\bbr e_{j},\qquad \bbe_{k}=\bbr e_{n-k+1} \oplus \cdots \oplus\bbr e_{n};\]
\[ \bbr^{n-j}=\bbr e_{j+1} \oplus \cdots \oplus\bbr e_{n},\qquad \bbr^{n-k}=\bbr e_1 \oplus \cdots \oplus\bbr e_{n-k}. \]
 Then
\bea
I&=&\frac{1}{\gam_{n-k}(\a)}\intl_{\bbr^{n-k}} |v|^{\a -n+k}dv\intl_{G_j (\bbr^{n-k})} d\xi \intl_{\bbe_k \cap \,\xi^\perp} f(\xi +u+v)\, du\nonumber\\
&=&\frac{1}{\gam_{n-k}(\a)}\intl_{\bbr^{n-k}} |v|^{\a -n+k}dv\intl_{O(n-k)} d\gam \intl_{\bbe_k} f(\gam(\bbe_j +u+v))\, du\nonumber\\
&=&\frac{1}{\gam_{n-k}(\a)}\intl_{\bbe_k} du\intl_0^\infty r^{\a -1} dr\intl_{\bbs^{n-k-1}}d\th \intl_{O(n-k)} f(\gam(\bbe_j +u+r\th))\, d\gam.\nonumber\eea
This expression can be transformed by making use of the bi-spherical coordinates \cite[p. 31]{Ru15}
\[\th=\vp \cos \om +\psi\sin \om, \quad \vp \in \bbs^{j-1}, \quad \psi \in \bbs^{n-k-j-1}, \quad 0<\om <\pi/2,\]
\[ d\th =\cos^{j-1} \om \,\sin^{n-k-j-1} \om  \,d\vp \, d\psi \, d\om. \]
We obtain
\bea I&=&\frac{1}{\gam_{n-k}(\a)}\intl_{\bbe_k} du\intl_0^\infty r^{\a -1} dr\intl_0^{\pi/2}\cos^{j-1} \om \,\sin^{n-k-j-1} \om\, d\om \nonumber\\
&\times& \intl_{\bbs^{j-1}}d\vp\intl_{\bbs^{n-k-j-1}}d\psi \intl_{O(n-k)} f(\gam(\bbe_j +u+r\psi\sin \om))\, d\gam.\nonumber\eea
The integral over $O(n-k)$  depends only on $r\sin \om$ and $u$. We denote it by $\tilde f (r\sin \om, u)$ and continue:
\bea
I \! \!&=&  \!\!\frac{\sig_{j-1}\,\sig_{n-k-j-1}}{\gam_{n-k}(\a)}\intl_{\bbe_k} \! du \!\intl_0^{\pi/2} \!\cos^{j-1} \om \,\sin^{n-k-j-1} \om\, d\om  \!\intl_0^\infty  \!r^{\a -1} \tilde f (r\sin \om, u)\,dr\nonumber\\
\label{jas70} &=&c_1 \intl_{\bbe_k} du \intl_0^\infty s^{\a -1} \tilde f (s, u)\,ds,\eea
where
\bea c_1&=&\frac{\sig_{j-1}\,\sig_{n-k-j-1}}{\gam_{n-k}(\a)} \intl_0^{\pi/2}\cos^{j-1} \om \,\sin^{n-k-j-\a-1} \om\, d\om\nonumber\\
&=&\frac{2^{1-a}\, \Gam((n-k-j-\a)/2)}{\Gam(\a/2)\, \Gam((n-k-j)/2)}.\nonumber\eea

Let us show that  $J$ has the form (\ref{jas70}) too. We have
\bea
J&=&\intl_{G_j (\bbr^{n-k})} d\xi\intl_{\bbe_k \cap \,\xi^\perp} (I_{n-j}^{\a}f)(\xi, u)\, du\qquad \text{\rm (note that $\bbe_k\cap \xi^\perp = \bbe_k$)}\nonumber\\
&=&\frac{1}{\gam_{n-j}(\a)}\intl_{O(n-k)} d\gam  \intl_{\bbe_k} du\intl_{\bbr^{n-j}} |y|^{\a -n+j} f(\gam(\bbe_j +u+y))\, dy\nonumber\\
&=&\frac{1}{\gam_{n-j}(\a)}\intl_{\bbe_k} du\intl_0^\infty r^{\a -1} dr\intl_{\bbs^{n-j-1}}d\th \intl_{O(n-k)} f(\gam(\bbe_j +u+r\th))\, d\gam.\nonumber\eea
Setting
\[\th=\vp \cos \om +\psi\sin \om, \quad \vp \in \bbs^{n-k-j-1}, \quad \psi \in \bbs^{k-1}, \quad 0<\om <\pi/2,\]
(cf. the first part of the proof), we continue
\bea
J&=&\frac{1}{\gam_{n-j}(\a)}\intl_{\bbe_k} du\intl_0^\infty r^{\a -1} dr\intl_0^{\pi/2}\cos^{n-j-k-1} \om \,\sin^{k-1} \om\, d\om \nonumber\\
&\times& \intl_{\bbs^{n-k-j-1}}d\vp\intl_{\bbs^{k-1}}d\psi \intl_{O(n-k)} f(\gam(\bbe_j +u+r\vp\, cos \,\om +r\psi\sin \om))\, d\gam.\nonumber\eea
The integral over $O(n-k)$  depends only on $r\cos \om$ and $u$.  We denote it by $\tilde f (r\cos \om, u)$. Then
\bea
J \! \!&=&  \!\!\frac{\sig_{k-1}\,\sig_{n-k-j-1}}{\gam_{n-j}(\a)}\intl_{\bbe_k} \! du \!\intl_0^{\pi/2} \!\cos^{n-j-k-1} \om \,\sin^{k-1} \om\, d\om  \!\intl_0^\infty  \!r^{\a -1} \tilde f (r\cos \om, u)\,dr\nonumber\\
\label{jas701} &=&c_2 \intl_{\bbe_k} du \intl_0^\infty s^{\a -1} \tilde f (s, u)\,ds,\eea
where
\bea c_2&=&\frac{\sig_{k-1}\,\sig_{n-k-j-1}}{\gam_{n-j}(\a)} \intl_0^{\pi/2}\cos^{n-j-k-\a-1}\om \,\sin^{k-1} \om\, d\om \nonumber\\
&=&\frac{2^{1-a}\, \Gam((n-k-j-\a)/2)}{\Gam(\a/2)\, \Gam((n-k-j)/2)}=c_1.\nonumber\eea
Comparing (\ref{jas70}) and (\ref{jas701}), we complete the proof.

\subsection{Fuglede Type Formulas}

We introduce the following integral operators acting on functions $h: \rn \to \bbc$:
\be\label{heit}
\Lam_{j,k} h =R_k^* \,\rjk R_j h, \qquad \Lam_{k,j} h =R_j^* \,\rkj R_k h.\ee
If $h$ is a radial function, $h(x)=h_0(|x|)$, then (\ref{pffrr}) together with (\ref{ppaawsdz}) and (\ref{ppaawsdz1}) gives

\bea
\Lam_{j,k} h &=&\frac{\tilde  c_1\,\pi^{j/2}\Gam (n/2)}{\Gam ((n-k)/2)}\, r^{2-n}\,I^{k/2}_{+,2} I^{j/2}_{+, 2} \, s^{n-j-k-2} I^{k/2}_{-, 2}I^{j/2}_{-,2} h_0\nonumber\\
&=&\frac{\tilde c_1\,\pi^{j/2}\Gam (n/2)}{\Gam ((n-k)/2)}\, r^{2-n}\,I^{(j+k)/2}_{+,2} \, s^{n-j-k-2} I^{(j+k)/2}_{-, 2} h_0.\nonumber\eea
By (\ref{rfa}), the last expression is a constant multiple of the Riesz potential $I_n^{j+k} h$.

This observation paves the way to the following general result.
\begin{theorem} \label {icgs9} If  $j+k<n$, then
\be\label{heit1} R_k^* \,\rjk R_j h \!=\! R_j^* \,\rkj R_k h \!=\! c \,I_n^{j+k} h, \quad c\!=\!\frac{2^{j+k} \pi^{(j+k)/2}\, \Gam (n/2)}{\Gam ((n-j-k)/2)},\ee
provided that the Riesz potential $I_n^{j+k} h$ exists in the Lebesgue sense.
\end{theorem}
\begin{proof}  Because all operators in (\ref{heit1}) commute with rigid motions and $j$ and $k$ are interchangeable,  it suffices to show that $(R_k^* \,\rjk R_j h)(0)= c \,(I_n^{j+k} h)(0)$.
By (\ref{mmsdcrt}) and (\ref{rpq}),
\bea\label{62}
&&(R_k^* \,\rjk R_j h) (0)=\intl_{O(n)} (\rjk R_j h)\,(\gam\bbe_k)\,d\gam\nonumber\\
&&=\intl_{O(n)}d\gam \intl_{G_j(\bbr^{n-k})}\,d \xi\intl_{\bbe_k} (R_j h)\,(\gam\xi +\gam y)\,dy\nonumber\\
&&=\intl_{O(n)}d\gam \intl_{O(n-k)}d\a \intl_{\bbe_k} (R_j h)\,(\gam\a\bbe_j+ \gam y)\,dy\nonumber\\
&&=\intl_{O(n)}d\gam\intl_{\bbe_k} (R_j h)\,(\gam (\bbe_j + y))\,dy.\nonumber\eea
Using (\ref{rtra1kfty}) (with $k$ replaced by $j$), we write the last expression as follows.
\bea
&&\intl_{O(n)}d\gam\intl_{\bbe_k} dy\intl_{\bbe_j} h(\gam (y+z))\, dz=\intl_{O(n)}d\gam\intl_{\bbe_j \oplus \,\bbe_k}  h(\gam \tilde y)\, d\tilde y\nonumber\\
&&=\frac{\sig_{j+k-1}}{\sig_{n-1}}\intl_0^\infty r^{j+k-1} dr \intl_{\bbs^{n-1}} h(r\th)\, d\th\nonumber\\
&&=\frac{\sig_{j+k-1}}{\sig_{n-1}} \intl_{\rn} h(x)\, |x|^{j+k-n}\, dx= c \,(I_n^{j+k} h)(0),\nonumber\eea
as desired.
\end{proof}

The formula (\ref{heit1}) is a generalization  of Fuglede's formula (\ref{kryafu}). The latter can be obtained from (\ref{heit1}) if we formally set $j=0$ or $k=0$.

\section{Inversion formulas for  $\rjk f$}

The following preliminary discussion explains the essence of the matter and the plan of the section. It might be natural to expect that $\rjk$ is injective on some standard function space, like $C_c^\infty  (\agnj)$, if $\dim \agnj \le \dim \agnk$, which is equivalent to $(j+1)(n-j)\le (k+1)(n-k)$. The latter splits in two cases:

\vskip 0.2 truecm

(a) $j+k=n-1$ for all $j$ and $k$;

(b) $j+k<n-1$ for $j\le k$.

\vskip 0.2 truecm

In the present paper we do not investigate both (a) and (b) in full generality and proceed as follows.  We first consider  $\rjk f$ for radial $f$,   when  an explicit  inversion formula is available for all $j+k<n$. Then we address to the case (a) and obtain an inversion formula
in  a sufficiently large class of functions $f$, including functions in Lebesgue spaces and continuous functions. The case of all $j+k<n$ for such functions remains open. Some progress can be achieved if
 we restrict the class of  functions $f$ to the  range of the $j$-plane transform. Under this assumption, $\rjk f$ can be explicitly inverted, no matter whether $j\le k$ or vice versa.

\subsection{The Radial Case}

If $f$ is   radial, then  $\rjk f$ is  radial too and  we have the following result.
\begin{theorem}\label {wiuytr}
 Let $f(\t)\equiv f_0 (|\t|)$   be a locally integrable  radial function on $\agnj$ satisfying (\ref{pfqq4e}). The function $f_0$ can be recovered from the Radon transform $(\rjk f)(\z)\equiv (I_{j,k} f_0)(|\z|)$ for all $j+k<n$ by the formula
\be\label{sjue}
f_0(t)=\tilde c_1{\!}^{-1}\,(\Cal D^{k/2}_{-, 2} \, r^{2-n+j+k} \,\Cal D ^{j/2}_{+, 2} \,s^{n-k-2}\,I_{j,k} f_0)(t),\ee
where  $\tilde c_1=\pi^{k/2}\, \Gam ((n-k)/2)/\Gam ((n-j-k)/2)$ and the  Erd\'{e}lyi--Kober fractional derivatives $\Cal D^{k/2}_{-, 2}$ and $ \Cal D ^{j/2}_{+, 2}$ are defined by  (\ref{frr+z})-(\ref{frr+z3}).
\end{theorem}
\begin{proof} By (\ref{pffrrt}),
\be\label{lrtvb}  (I_{j,k} f_0)(s)\!=\!\tilde c_1\,s^{k+2-n} (I^{j/2}_{+, 2} \, r^{n-j-k-2} I^{k/2}_{-, 2} f_0)(s).\ee
The assumption (\ref{pfqq4e}) is equivalent to $\int_a^\infty |f_0 (t)|\, t^{k-1}dt <\infty$ for some $a>0$. The local integrability of $f$ implies
 $r^{n-j-k-1} I^{k/2}_{-, 2} f_0 \in L^1_{loc} (\bbr_+)$ (a simple calculation is left to the reader). Hence the conditions of Theorems \ref{78awqe555} and  \ref{78awqe} are satisfied and  both fractional integrals in (\ref{lrtvb}) can be inverted to give (\ref{sjue}).
\end{proof}

Interchanging $j$ and $k$, the reader can easily  obtain a similar statement  for the dual transform $\rkj \vp$.

\subsection{The  Case $j+k=n-1$}

\subsubsection{The Structure of $\rjk f$}
 We consider the flag
\[ \frF =\{(\xi, \eta, v): \xi \in \gnj, \; \eta \in G_k (\xi^\perp), \; v\in \xi^\perp \cap \eta^\perp\}.\]
 Given a function $f$ on $\agnj$, we define a function $\tilde f$ on $\frF$  by the formula
\[ \tilde f (\xi, \eta, v)=\intl_{\eta}f(\xi+u+v)\,d_\eta u.\]
 This function is the inner integral in  the definition (\ref {rpq}) of $\rjk f$ and has two interpretations. On the one hand, for every fixed $\xi \in \gnj$, $\tilde f$ is the $k$-plane transform of the function $u\to f(\xi,u)$ in the $(n-j)$-space $\xi^\perp$ :
\be\label{ksh}
\tilde f_\xi (\z)\equiv  \tilde f (\xi, \eta, v)=(R_{k, \xi^\perp}[f(\xi, \cdot)])(\z), \quad \z=\z(\eta ,v) \in G(k, \xi^\perp).\ee
On the other hand, for every fixed $\eta\in \gnk$, $\tilde f$ is a function on the affine Grassmannian $G(j, \eta^\perp)$:
\be\label{ksh1}
\tilde f_\eta (\t')\equiv  \tilde f (\xi, \eta, v)=\intl_\eta f(\t'+u)\, du, \quad \t'=\t'(\xi,v) \in G(j, \eta^\perp).\ee
The dual $j$-plane transform of $\tilde f_\eta (\cdot)$ in the  $(n-k)$-space $\eta^\perp$  at the point $v\in \eta^\perp$ has the form
\bea
\label{ksh2}&&(R^*_{j, \eta^\perp}\tilde f_\eta)(v)  =\intl_{O( \eta^\perp)} \tilde f_\eta (\gam\xi+v )\, d\gam\\
&&=\intl_{G_j(\eta^{\perp})}\,d_{\eta^{\perp}} \xi\intl_{\eta}f(\xi+u+v)\,d_\eta u= (\rjk f)(\eta,v),\nonumber
\eea
where $O( \eta^\perp)$ is the subgroup of $O(n)$, which consists of orthogonal transformations in $\eta^\perp$.

The above reasoning shows that $\rjk$ is a certain mixture (but not a composition) of the $k$-plane transform and the dual $j$-plane transform.

\subsubsection{Inversion  Procedure}

According to the structure of the operator $\rjk$, to reconstruct $f$ from $\rjk f$,  we  first
 invert the dual  $j$-plane transform (\ref  {ksh2}) in the $(n-k)$-space $\eta^\perp$ and then the $k$-plane transform (\ref  {ksh})  in the $(n-j)$-space $\xi^\perp$. Because,  in general,  the dual  $j$-plane transform  is injective only in the co-dimension one case (cf. \cite[Theorem 4.4]{Ru04a}),  we restrict  to the case $j=n-k-1$, when the injectivity can be proved on a pretty  large class of functions.

\vskip 0.2 truecm

\noindent STEP 1.  We make use of Corollary \ref{cord} with $\rn$ replaced by $\eta^\perp$, when the condition (\ref{lafit2d}) becomes
\be\label{slyyy}
\intl_{G(j,\eta^\perp)}\frac{|\tilde f_\eta(\t')|}{1+|\t'|}\,d\t'<\infty\quad \text {\rm for almost all $\eta\in \gnk$}.\ee
Under this condition,
\be\label{slvkjn3} \tilde f_\eta (\t')=(R^*_{\eta^\perp})^{-1} [(\rjk f)(\eta, \cdot)](\t'), \quad \t'= \xi+v \in G(j, \eta^\perp),\ee
where $(R^*_{\eta^\perp})^{-1}$ stands for the inverse dual  $j$-plane transform  in the $(j+1)$-subspace $\eta^\perp$. An explicit formula for $(R^*_{\eta^\perp})^{-1}$ can be obtained from the equality (\ref{lafit77}) adapted for our case.

Our next aim is to find sufficient conditions for (\ref{slyyy}) in terms of $f$. We observe that (\ref{slyyy}) will be proved if we show that
\[ I=\intl_{\gnk}d\eta\intl_{G_j(\eta^{\perp})}d_\eta\xi\intl_{\xi^{\perp}\cap\eta^{\perp}}\frac{|\tilde f_\eta(\xi,v)|}{1+|v|}\,dv<\infty\,.
 \]
Changing the order of integration and using (\ref{ksh}),
we can write $I$ as
\[ I=\intl_{\gnj}d\xi\intl_{G_k(\xi^{\perp})}d_\xi\eta\intl_{\xi^{\perp}\cap\eta^{\perp}}\frac{|(R_{k, \xi^\perp}[f(\xi, \cdot)])(\eta,v)|}{1+|v|}dv,\]
where $R_{k, \xi^\perp}$ stands for the $k$-plane transform in the $(k+1)$-space $\xi^\perp$.  Using (\ref{SDIvyh34}) with $n=k+1$, we have
\[\intl_{G_k(\xi^{\perp})}d_\xi\eta\intl_{\xi^{\perp}\cap\eta^{\perp}}\frac{|(R_{k, \xi^\perp}[f(\xi, \cdot)])(\eta,v)|}{1+|v|}dv\le
 c\, \intl_{\xi^{\perp}} \frac{|f(\xi, u)|}{1+|u|}\, \log (2+|u|)\, du,\]
whence
\[ I\le  c\, \intl_{\gnj}d\xi\intl_{\xi^{\perp}} \frac{|f(\xi, u)|}{1+|u|}\, \log (2+|u|)\, du =c\intl_{\agnj}\frac{|f(\t)|}{1+|\t|}\, \log (2+|\t|)\, d\t.\]
Thus the inversion formula (\ref{slyyy}) is valid if
\be\label{hasde}
\intl_{\agnj}\frac{|f(\t)|}{1+|\t|}\, \log (2+|\t|)\, d\t <\infty.\ee

\begin{definition} The class of all functions $f$ satisfying (\ref{hasde}) will be denoted by $L_{\log}(\agnj)$.

\end{definition}

\vskip 0.2 truecm

\noindent STEP 2. By Corollary \ref{corr}, the function  $f(\xi,\cdot)$ on  $\xi^\perp$ can be reconstructed from its $k$-plane transform $\tilde f_\xi (\z)=(R_{k, \xi^\perp}[f(\xi, \cdot)])(\z)$ if
 \[ \intl_{\xi^{\perp}} \frac{|f(\xi,u)|}{1+|u|} d\,u\,<\infty .
  \]
  The latter is guaranteed for almost all $\xi$  if
\be\label{hasdeyy}
\intl_{\agnj} \frac{|f(\t)|}{1+|\t|} d\t <\infty.\ee
Thus Step 2 gives one more assumption for $f$ which  is, however, weaker than (\ref{hasde}).

Combining Step 1 and Step 2, we arrive at the following statement.
\begin{theorem} \label {kwwte}
If $j+k=n-1$, then  every function $f\in L_{\log}(\agnj)$  can be reconstructed from $\vp=\rjk f$ by
the formula
\be\label{slvkjn11}
f(\xi, u)=(R^{-1}_{k, \xi^\perp}[\tilde f_\xi])(u),    \ee
where
\be\label{slvkjn12}  \tilde f_\xi (\eta, v)\equiv \tilde f_\eta (\xi, v)=(R^*_{\eta^\perp})^{-1} [\vp(\eta, \cdot)](\xi,v),\ee
with the inverse transforms $R^{-1}_{k, \xi^\perp}$ and $(R^*_{\eta^\perp})^{-1}$ being defined  according to Corollaries \ref{corr}  and \ref{cord}, respectively.
\end{theorem}

\begin{remark} The condition $f\in L_{\log}(\agnj)$ in Theorem \ref{kwwte}
falls into the scope of the Existence Theorem \ref {kwwt} and differs from the latter only by the logarithmic factor. Thus
Theorem \ref{kwwte}  provides inversion of $\rjk f$ under almost minimal assumptions. Note also that
by H\"{o}lder's inequality, any function in  $L^p(\agnj)$, $1\le p<(n-j)/k$, and any continuous function of order $O(|\t|^{-\mu})$, $\mu>k$,
belong to $L_{\log}(\agnj)$, where the bounds for $p$ and $\mu$ are sharp; cf.   (\ref{lehget}).

\end{remark}

\subsection{Inversion of $\rjk f$ on the Range of the $j$-Plane Transform} \label {kslfeyut}

Theorem \ref{icgs9} implies the following inversion result, which resembles Corollary \ref{kryy6} for the dual $k$-plane transform.

\begin{theorem} \label{qqww}  Let $f \!=\! R_j h$, $h\in L^p (\rn)$. If  $1\le p <n/(j+k)$, then
\be\label{ktyu88} f\! =\! c^{-1} R_j \bbd_n^{j+k} R_k^*\, \rjk f,  \qquad c\!=\!\frac{2^{j+k} \pi^{(j+k)/2}\, \Gam (n/2)}{\Gam ((n-j-k)/2)}, \ee
where $\bbd_n^{j+k}$ is the Riesz fractional derivative (\ref{sdyt}). More generally, if $0<\a<n-j-k$ and  $1\le p <n/(j+k+\a)$,
 then
\be\label{ktyu88a} f\! =\! c^{-1} R_j \bbd_n^{j+k+\a} R_k^*\, I_{n-k}^\a  \rjk f= c^{-1} R_j \bbd_n^{j+k+\a}  R_k^*\,\rjk  I_{n-j}^\a f\ee
with the same constant $c$.
\end{theorem}
\begin{proof} By (\ref{heit1}),
\[
c^{-1} R_j \bbd_n^{j+k} R_k^* \rjk f=c^{-1} R_j \bbd_n^{j+k} R_k^* \rjk R_j h=R_j \bbd_n^{j+k} I_n^{j+k} h= R_j h=f.\]
Further, combining (\ref{heit1}) with the semigroup property of Riesz potentials, we obtain
\[
 I_n^\a R_k^* \,\rjk f \!=\! I_n^\a  R_k^* \,\rjk R_j h = I_n^\a  I_n^{j+k} h =   I_n^{j+k+\a} h.\]
However, by (\ref{ktyu}) and (\ref{wleiy}),
\[ I_n^\a  R_k^* \,\rjk f= R_k^* I_{n-k}^\a  \,\rjk f =R_k^*   \,\rjk I_{n-j}^\a f.\]
 This gives (\ref{ktyu88a}).
\end{proof}

The formula (\ref{ktyu88a}) can be used if we want to replace the nonlocal  Riesz fractional derivative by the local one. For example, if
$j+k$ is odd and  $j+k<n-1$, we can apply (\ref{ktyu88a}) with $\a=1$.

\begin{remark}  If $f$ is good enough, then (\ref{ktyu88}) formally agrees with the inversion formula of Gonzalez \cite[Theorem 3.4]{Go87}. In our notation his formula reads
\be\label{ktyu88Go} f \!=\! \tilde c^{-1} \rkj \bbd_{n-k}^{n-1}  \rjk f, \quad \tilde c\!=\!2^{n-1} \pi^{(n-3)/2} \Gam \left(\frac{j\!+\!1}{2}\right) \Gam \left(\frac{k\!+\!1}{2}\right). \ee
The following non-rigorous reasoning shows  the consistency of (\ref{ktyu88}) and (\ref{ktyu88Go}). It  suffices to show that
\be\label{ktyu88Go1}
c^{-1} R_j \bbd_n^{n-1} R_k^*\vp = \tilde c^{-1} \,\rkj \bbd_{n-k}^{n-1} \vp,\ee
 where $c$ is the constant from (\ref{ktyu88}) with $j+k=n-1$. We set $\vp=R_k h$ and apply $R_j^*$ to both sides of (\ref{ktyu88Go1}). For the left-hand side,  (\ref{kryafu}) yields
\bea
&&c^{-1} R_j^*R_j \bbd_n^{n-1} R_k^* R_k h = c^{-1}c_{j,n}c_{k,n} I_n^{j}\,\bbd_n^{n-1} I_n^{k} h\nonumber\\
\label {lsf7t}&&=\frac{\pi^{1/2}
 \Gam (n/2)}{ \Gam ((n-j)/2)\, \Gam ((n-k)/2)}\, h.\eea
For the right-hand side,  the  formula   $\bbd_{n-k}^{n-1} R_k h=R_k\bbd_{n}^{n-1} h $ (cf. (\ref{ktyu})) in conjunction with    (\ref{heit1}) gives
\[
\tilde c^{-1} R_j^*  \rkj \bbd_{n-k}^{n-1} R_k h =\tilde c^{-1}R_j^*  \rkj R_k  \bbd_{n}^{n-1} h =c_1\, I_n^{n-1}  \bbd_{n}^{n-1} h =c_1\, h,\]
where $c_1$ is exactly the same as in (\ref{lsf7t}).
Thus, if $R_j^*$ is injective and the class of functions $f$ is good enough, we are done.

Note that the proof of convergence of the expression on right-hand side of (\ref{ktyu88Go}) is  rather nontrivial, even for $f\in C_c^\infty (\agnj)$.
\end{remark}

\section{Appendix }

{\it Proof of  (\ref{puor}).}  Let us show that  (\ref{puor}) follows from  (\ref{pf4e}). It suffices to assume $f_0 \ge 0$.
 We recall that $\ell =n-j-k\ge 1$ and the letter $c$ stands  for a constant that can be different  at each occurrence. Let
\[F(r)=\intl_r^{\infty}\!f_0(t)\,(t^2\!-\!r^2)^{k/2-1} t\,dt\,.\]
For any $0<\a<\b<\infty$ we have
\bea\label{ap1}
&&\intl_\a^\b (I_{j,k} f_0)(s)\,ds\le c \intl_\a^\b ds\intl_0^s(s^2\!-\!r^2)^{j/2-1} r^{\ell-1}
F(r)\,dr\nonumber\\
&&\le c\intl_\a^\b ds\intl_0^s(s\!-\!r)^{j/2-1} r^{\ell-1}F(r)\,dr\nonumber\\
&&= c\intl_0^\a  r^{\ell-1} F(r)\, dr\intl_\a^\b (s-\!r)^{j/2-1}ds+c\intl_\a^\b  r^{\ell-1}F(r)  dr\intl_r^\b (s-\!r)^{j/2-1}ds\nonumber\\
&&\label {lare}\le c\intl_0^\b  r^{\ell-1}  (\b-\!r)^{j/2}F(r)\,dr-c\intl_0^\a  r^{\ell-1}  (\a-\!r)^{j/2}F(r)\,dr.
\eea
Because the  integrals in (\ref{lare}) have the same form, it suffices to show that
\bea
I(\a)\equiv \intl_0^\a  r^{\ell-1}  (\a-\!r)^{j/2}F(r)\,dr<\infty\,\qquad \forall \;\a>0.\nonumber
\eea
We have
\bea
&&I(\a)=\intl_0^\a  r^{\ell-1}  (\a-\!r)^{j/2}dr\intl_r^{\infty}\!f_0(t)(t-\!r)^{k/2-1} t^{k/2} dt\nonumber\\
&&\le c\intl_0^\a \!f_0(t)\,t^{k/2+\ell-1}dt\! \intl_0^t(t\!-\!r)^{k/2-1} dr\!+\!c\!\intl_\a^\infty \! f_0(t)\,t^{k/2} dt\!\intl_0^\a (t\!-\!r)^{k/2-1} dr\nonumber\\
&&=c \intl_0^\a f_0(t)t^{k+\ell-1}dt+c\intl_\a^\infty  f_0(t)\,t^{k/2} (t^{k/2}-(t-\a)^{k/2}) \,dt\nonumber\\
&&\le c\intl_0^\a f_0(t)t^{n-j-1}dt+c\intl_\a^\infty  f_0(t)\,t^{k-1} dt<\infty\,.\nonumber
\eea
The last expression is finite by  (\ref{pf4e}).

\end{document}